\documentclass[10pt,reqno]{amsart}
\usepackage{amsmath,amsthm,amsfonts,amscd,amssymb,latexsym,mathrsfs}
\usepackage{comment}
\usepackage{graphicx}
\usepackage[pdfusetitle,colorlinks]{hyperref}
\PassOptionsToPackage{dvipsnames}{xcolor}
\hypersetup{citecolor=OliveGreen,linkcolor=Mahogany,urlcolor=Plum}
\usepackage[utf8]{inputenc}
\usepackage[capitalise]{cleveref}
\usepackage{enumerate}
\theoremstyle{plain}
\newtheorem{theorem}{Theorem}[section]
\newtheorem{corollary}[theorem]{Corollary}
\newtheorem{lemma}[theorem]{Lemma}
\newtheorem{proposition}[theorem]{Proposition}
\newtheorem{thmy}{Theorem}
 
\newenvironment{thmx}{\stepcounter{theorem}\begin{thmy}}{\end{thmy}}
\theoremstyle{definition}
\newtheorem{remark}[theorem]{Remark}
\newtheorem{example}[theorem]{Example}

\newtheorem{definition}[theorem]{Definition}

\usepackage{combelow}

\newcommand{\Z}{\mathbf{Z}}

\newcommand{\R}{\mathbf{R}}
\newcommand{\Q}{\mathbf{Q}}
\newcommand{\C}{\mathbf{C}}

\newcommand{\frakm}{\mathfrak{m}}

\newcommand{\ord}{\textrm{ord}}

\renewcommand{\H}{\mathcal{H}}

\newcommand{\ds}{\displaystyle}

\usepackage{tikz}
\usetikzlibrary{matrix}

\usepackage{caption}
\usepackage{subcaption}

\numberwithin{equation}{section}

\crefname{section}{\S}{\S\S}
\Crefname{section}{\S}{\S\S}

\crefformat{equation}{(#2#1#3)}

\makeatletter
\def\@tocline#1#2#3#4#5#6#7{\relax
  \ifnum #1>\c@tocdepth 
  \else
    \par \addpenalty\@secpenalty\addvspace{#2}%
    \begingroup \hyphenpenalty\@M
    \@ifempty{#4}{%
      \@tempdima\csname r@tocindent\number#1\endcsname\relax
    }{%
      \@tempdima#4\relax
    }%
    \parindent\z@ \leftskip#3\relax \advance\leftskip\@tempdima\relax
    \rightskip\@pnumwidth plus4em \parfillskip-\@pnumwidth
    #5\leavevmode\hskip-\@tempdima
      \ifcase #1
       \or\or \hskip 1em \or \hskip 2em \else \hskip 3em \fi%
      #6\nobreak\relax
    \dotfill\hbox to\@pnumwidth{\@tocpagenum{#7}}\par
    \nobreak
    \endgroup
  \fi}

\renewcommand{\O}{\mathcal{O}}

\newcommand{\J}{\mathcal{J}}

\newcommand{\Gal}{\textrm{Gal}}

\DeclareMathOperator{\Ends}{Ends}
\DeclareMathOperator{\mass}{mass}
\newcommand{\M}{\mathcal{M}}

\DeclareMathOperator{\rig}{rig}

\newcommand{\hotimes}{\widehat{\otimes}}

    \usepackage{soul}
    
\usepackage{mathtools}

\author{Matthew Stevenson}
\title{A Non-Archimedean Ohsawa--Takegoshi Extension Theorem}
\date{\today}
\address{Department of Mathematics, University of Michigan, Ann Arbor, MI
48109--1043, USA}
\email{\href{mailto:stevmatt@umich.edu}{stevmatt@umich.edu}}
\urladdr{\url{http://www-personal.umich.edu/~stevmatt/}}
\begin{document}

\maketitle

\begin{abstract}
We prove an Ohsawa--Takegoshi-type extension theorem on the Berkovich closed unit disc over certain non-Archimedean fields.\ As an application, we establish a non-Archimedean analogue of Demailly's regularization theorem for quasisubharmonic functions on the Berkovich unit disc.
\end{abstract}

\section{Introduction}\label{section:intro}

The Ohsawa--Takegoshi theorem is one of the fundamental extensions theorems in complex geometry. Originating in the foundational paper~\cite{ohsawa-takegoshi} of Ohsawa--Takegoshi, many generalizations and improvements have since been shown; see~\cite{manivel1993theoreme,berndtsson1996extension,siu1996fujita,demailly2000ohsawa,mcneal-varolin}. 
Its many applications include Siu's much-celebrated proof~\cite{siu1998invariance} of the deformation invariance of plurigenera.
In its simplest form, the classical Ohsawa--Takegoshi theorem asserts the following: given a plurisubharmonic function $\varphi$ on the complex unit disc $\mathbf{D}$, a point $z \in \mathbf{D} \backslash \{ \varphi = -\infty \}$, and a value $a \in \C$, there is a holomorphic function $f$ on $\mathbf{D}$ such that $f(z) = a$ and 
$$
\int_{\mathbf{D}} |f(x)|^2 e^{-2\varphi(x)} d\lambda \leq \pi|f(z)|^2 e^{-2\varphi(z)}.
$$
The constant $\pi$ is optimal, as shown in~\cite[Theorem 1]{blocki}. There is also an adjoint formulation of the result, which concerns the extension of a holomorphic 1-form rather than of a holomorphic function.

In~\cite{berkovich,berkovich93}, Berkovich introduced his theory of analytic spaces over a complete non-Archimedean field. 
In recent years, analogues of many theorems from complex geometry have been developed in the non-Archimedean setting, notably the solution in~\cite{bfj15} of a non-Archimedean Calabi--Yau-type problem. It is natural to ask for an analogue of the Ohsawa--Takegoshi extension theorem in the non-Archimedean setting, as well.

Let $k$ be a field, complete with respect to a non-Archimedean absolute value $| \cdot |$. Let $k\{T \}$ denote the Tate algebra in one variable over $k$, $X = E_k(1)$ denote the Berkovich closed unit disc over $k$, and $A_X \colon X \to [0,+\infty]$ denote Temkin's canonical metric on the canonical bundle $\omega_{X/k}$ (more concretely, $A_X$ is $-\log$ of the radius function). For any quasisubharmonic function $\varphi$ on $X$ and analytic function $f \in k\{ T\}$, consider the norm
$$
\| f \|_{\varphi} \coloneqq \sup_{X \backslash Z(\varphi)} |f|e^{-\varphi - A_X} 
$$
where $Z(\varphi ) \coloneqq \{ \varphi = -\infty \}$. The function $\varphi + A_X$ can be thought of as a metric on $\omega_{X/k}$ and $f$ as a global section of $\omega_{X/k}$, so the norm $\| f \|_{\varphi}$ measures the length of the section $f$ in the metric $\varphi + A_X$. See~\S\ref{section:preliminaries} for further details.

Our first main result is a non-Archimedean version of the Ohsawa--Takegoshi extension theorem on the Berkovich closed unit disc $X = E_k(1)$.

\begin{thmx}\label{thm:main}
Assume $k$ is algebraically closed, trivially-valued, or is spherically complete of residue characteristic zero. 
Let $\varphi$ be a quasisubharmonic function on $X = E_k(1)$.
For any $z \in X$, there exists a nonzero polynomial $f \in k[T]$ such that
$$
\lim_{\epsilon \to 0^+} \| f \|_{(1+\epsilon)\varphi} \leq |f(z)| e^{-\varphi(z)}.
$$
If $\varphi(z) = -\infty$, then we may find $f$ such that $\lim_{\epsilon \to 0^+} \| f\|_{(1+\epsilon)\varphi} < +\infty$. Moreover, if $k$ is algebraically closed and $z$ is a rigid point of $X$, then for any value $a \in \H(z)^*=k^*$, we may find $f$ such that $f(z) = a$.
\end{thmx}

To clarify the relationship between~\cref{thm:main} and the adjoint formulation of the Ohsawa--Takegoshi extension theorem, we must discuss the non-Archimedean analogues of (pluri)subharmonic functions and volume forms, which arise in the statement of the latter. We are interested in the class of quasisubharmonic functions on the Berkovich closed unit disc, which are the non-Archimedean analogue of (pluri)subharmonic functions on the complex unit disc. These are briefly discussed in~\S\ref{section:quasisubharmonic}; see~\cite{baker-rumely, dynberko} for a comprehensive treatment. 

On complex manifolds, volume forms naturally correspond to smooth metrics on the canonical bundle. 
The non-Archimedean analogue of a metric on the canonical bundle has been examined by Temkin in~\cite{temkin}. 
Temkin produces a canonical lower-semicontinuous metric $A_X$ on the canonical bundle $\omega_{X/k}$ of a $k$-analytic space $X$, extending the weight function of Musta\cb{t}\u{a}--Nicaise~\cite{mustata-nicaise}.
Temkin's metric is discussed in~\S\ref{section:temkin}. This is the natural candidate to replace the volume form that appears in the statement of the classical Ohsawa--Takegoshi theorem. 

In~\cref{thm:main}, a section is measured using the sequence of norms $\| \cdot \|_{(1+\epsilon)\varphi}$ as $\epsilon \to 0^+$ instead of with the single norm $\| \cdot \|_{\varphi}$, which is what one might expect from the classical Ohsawa--Takegoshi theorem. The former proves to give the correct analogy with the complex setting, as the following example demonstrates. Consider the (pluri)subharmonic function $\varphi = \alpha \log |z|$, with $\alpha > 0$, on the complex unit disc $\mathbf{D}$; then, it is elementary to check that $\int_{\mathbf{D}} e^{-2\varphi} d\lambda <+\infty$ if and only if $\alpha < 1$. Similarly, consider the quasisubharmonic function $\varphi = \alpha \log |T|$, with $\alpha > 0$, on the Berkovich unit disc $X = E_k(1)$; then, $\lim_{\epsilon \to 0^+} \| 1 \|_{(1+\epsilon)\varphi} < +\infty$ if and only if $\alpha < 1$.

As an application of \cref{thm:main}, we prove a regularization theorem for quasisubharmonic functions on the Berkovich unit disc. Certain results in this direction already exist in the literature: in~\cite[Theorem 2.10]{dynberko}, it was shown that any quasisubharmonic function on $X$ is the decreasing limit of bounded quasisubharmonic functions. A similar argument appears in~\cite[\S4.6]{favre06}. However, 
these constructions use
only the tree structure on the Berkovich unit disc (in particular, they do not incorporate the analytic structure). It is therefore unlikely that these proofs can be generalized to higher dimensions. 

As inspiration, we use the much-celebrated regularization theorem of Demailly for a plurisubharmonic function $\phi$ on a bounded pseudoconvex domain $\Omega \subset \C^n$. To a plurisubharmonic function $\phi$ on $\Omega$ and a positive integer $m$, we associate the Hilbert space $\mathscr{H}_{m\phi}$ of holomorphic functions on $\Omega$ satisfying the integrability condition 
$$\int_{\Omega} |f|^2 e^{-2m\phi} d\lambda\ < +\infty.$$
The Demailly approximation associated to $\mathscr{H}_{m\phi}$ is a plurisubharmonic function $\phi_m$ on $\Omega$ with analytic singularities, 
and the sequence $(\phi_m)_{m=1}^{\infty}$ converges pointwise and in $L^1_{\textrm{loc}}$ to $\phi$.
See~\cite{demailly92} for further details. 

We adopt the same philosophy in the non-Archimedean setting: to a quasisubharmonic function $\varphi$ on $X$, we associate the ideal $\H_{\varphi}$ of the Tate algebra $k\{ T \}$ consisting of those analytic functions $f$ satisfying the finiteness condition
$$
\| f \|^+_{\varphi} \coloneqq \lim_{\epsilon \to 0^+} \sup_{X \backslash Z(\varphi)} |f|e^{-(1+\epsilon)\varphi - A_X} < + \infty
$$
For each positive integer $m$, we define the non-Archimedean Demailly approximation $\varphi_m$ by the formula 
$$
\varphi_m \coloneqq \frac{1}{m} \left( \sup_{f \in \H_{m\varphi} \backslash \{ 0 \}} \log \frac{|f|}{\| f\|^+_{m\varphi}} \right)^*,
$$
where $( - )^*$ denotes the upper-semicontinuous regularization. We show that $\varphi_m$ is a quasisubharmonic function with analytic singularities on $X$. 

The ideal sheaf on $X$ associated to $\H_{\varphi}$ behaves like a non-Archimedean multiplier ideal associated to $\varphi$, and may be of independent interest. This idea is briefly explored in~\S\ref{section:multiplier},\ where we show that $\H_{\varphi}$ generates the stalks of a locally-defined multiplier ideal sheaf associated to $\varphi$. These multiplier ideals are used to show that the ideals $\H_{\varphi}$ satisfy a subadditivity property. 

We prove the following non-Archimedean analogue of Demailly's regularization theorem.

\begin{thmx}\label{thm:application}
Assume $k$ is algebraically closed, trivially-valued, or is spherically complete with residue characteristic zero. 
If $\varphi$ is a quasisubharmonic function on $X$ with $\varphi \leq 0$, then $\varphi \leq \varphi_m \leq \varphi + \frac{A_X}{m}$. In particular, the sequence $(\varphi_m)_{m=1}^{\infty}$ converges pointwise to $\varphi$ on $\{ A_X < +\infty \} \subseteq X$. 
\end{thmx}

In~\cref{thm:application}, the crucial inequality $\varphi_m \geq \varphi$ is a consequence of~\cref{thm:main}. In principle, a statement similar to~\cref{thm:application} ought to be possible in higher dimensions, but this would likely require a higher-dimensional version of the Ohsawa--Takegoshi theorem on analytic spaces.

Nonetheless, there has been much work done on the development of pluripotential theory on analytic spaces. Quite generally, Chambert-Loir and Ducros have introduced in~\cite{clducros} the notion of \emph{continuous} plurisubharmonic functions. In addition, semipositive metrics on line bundles were studied in detail by~\cite{zhang1995small, gubler1998local}, among others. On analytic curves, potential theory is well-established, due to the work of Thuillier~\cite{thuillier2005theorie}. 
A regularization theorem similar to~\cref{thm:application} is proven in the higher-dimensional setting in~\cite[Theorem B]{bfj16}. A related discussion appears in~\cite[\S 5]{bfj08}. 

The techniques involved in the proof of~\cref{thm:main} rely crucially on the tree structure of the Berkovich unit disc, and thus the proof does not, a priori, generalize to higher dimensions. More precisely, the proof proceeds by first constructing a finite subtree $\Gamma_{\varphi}$ of $X$ that captures the worst of the singularities of $\varphi$, and reducing to proving~\cref{thm:main} on the convex hull of $\Gamma_{\varphi} \cup \{ z \}$. For some end of the tree $\Gamma_{\varphi}$, we construct a new quasisubharmonic function $\phi$ such that $\Gamma_{\phi} \subsetneq \Gamma_{\varphi}$ and reduce to proving~\cref{thm:main} for $\phi$. This inductively reduces~\cref{thm:main} to a simple case, which is solved directly.

The paper is organized as follows. In~\S\ref{section:preliminaries}, we briefly review the definition of the Berkovich unit disc, and its potential theory. In~\S\ref{section:temkin}, we discuss Temkin's metric on the canonical bundle. In~\S\ref{section:ot}, we prove \cref{thm:main} and a variant thereof. In~\S\ref{section:application}, we construct the ideal $\H_{\varphi}$ and the non-archimedean Demailly approximation $\varphi_m$ associated to a quasisubharmonic function $\varphi$ and a positive integer $m \geq 1$. These are then used to prove \cref{thm:application}.\\

\noindent\textbf{Acknowledgements}. I would like to thank my advisor, Mattias Jonsson, for suggesting the problem and for his invaluable help, guidance, and support.\
I would also like to thank Kiran Kedlaya, J\'er\^{o}me Poineau, and Daniele Turchetti for helpful conversations regarding~\cref{multiplicity_descent}.\
I am grateful to Takumi Murayama and Emanuel Reinecke for their many comments on a previous draft. 
Finally, I would like to thank the anonymous referee for their many helpful comments, and for pointing out an error in a previous version.
This work was partially supported by NSF grant DMS-1600011.

\section{Preliminaries}\label{section:preliminaries}

Let $k$ be a field complete with respect to a (possibly trivial) non-Archimedean absolute value $| \cdot |$, and let $K \coloneqq \widehat{k^a}$ denote its completed algebraic closure. Let $k^{\circ} \coloneqq \{ | \cdot | \leq 1 \}$ denote the valuation ring of $k$, and $K^{\circ}$ the valuation ring of $K$. 

To establish conventions, we briefly recall in~\S\ref{section:disc} the definition of the Berkovich unit disc over a (not necessarily algebraically closed) non-Archimedean field $k$, and describe its points; in~\S\ref{section:quasisubharmonic}, we review the metric tree structure and potential theory on the Berkovich unit disc. Finally, in~\S\ref{section:temkin}, we discuss the canonical metric of Temkin on the canonical bundle of a $k$-affinoid space.

\subsection{The Berkovich unit disc}\label{section:disc}
The \emph{Tate algebra} $k\{ T \}$ in the variable $T$ is the subalgebra of $k[[T]]$ consisting of those power series $f = \sum_{i=0}^{\infty} a_i T^i$, with $a_i \in k$, such that $|a_i| \to 0$ as $i \to \infty$. The Tate algebra is a Banach $k$-algebra when equipped with the \emph{Gauss norm}
$$
|f(x_G)| \coloneqq \max_{i \geq 0} |a_i |.
$$
See~\cite[\S5.1]{bgr} for further details.

The \emph{Berkovich unit disc} is the set $X \coloneqq \M(k\{ T \})$ of multiplicative seminorms on the Tate algebra $k\{ T \}$ which extend the absolute value on $k$ and which are bounded above by the Gauss norm $x_G$.
When equipped with the topology of pointwise convergence, $X$ is a compact Hausdorff path-connected space. 
Let $\H(x)$ denote the completed residue field of $x \in X$.

Define a partial order $\leq $ on $X$ by declaring that $x \leq y$ if and only if $|f(x)| \leq |f(y)|$ for all $f \in k\{ T \}$; in this way, the pair $(X, \leq )$ becomes a rooted tree with root at the Gauss point $x_G$. The tree structure on $X$ is discussed in more detail in~\cref{section:quasisubharmonic}. 

Let $\Gal(k^a/k)$ denote the group of automorphisms of $k^a$ fixing $k$ (though $k^a/k$ is not Galois when $k$ is not perfect), and let $X_K \coloneqq \M(K \{ T \})$ denote the ground field extension of $X$ to $K$, which comes equipped with a continuous surjective map $\pi \colon X_K \to X$. The ground field extension $X_K$ carries a $\Gal(k^a/k)$-action, extending the natural action on $K$ by isometries, such that $\pi$ induces a homeomorphism $X_K / \Gal(k^a/k) \stackrel{\sim}{\longrightarrow} X$, where $X_K /\Gal(k^a/k)$ has the quotient topology. See~\cite[p.18]{berkovich} and~\cite[\S1]{baker-rumely} for further details.

Given $z \in K^{\circ}$ and $r \in [0,1]$, we may construct a point $x_{z,r} \in X_K$ as the sup-norm over the closed disc $\overline{D}(z,r) \subset K^{\circ}$ of radius $r$ about $z$; that is,
$$
|f(x_{z,r})| \coloneqq \sup_{z' \in \overline{D}(z,r)} |f(z')|, \textrm{\hspace{5mm} $f \in K\{ T \}$.}
$$ 
The point $\pi(x_{z,r}) \in X$ will again be denoted by $x_{z,r}$. When $K \not= k$, it is possible that two pairs $(z,r)$ and $(z',r')$ may define the same point of $X$; for example, if $z' \in \Gal(k^a/k) \cdot \{ z \}$ and $r > 0$, then $x_{z',r} = x_{z,r}$.

When $z \in (k^a)^{\circ}$ and $r = 0$, the associated point $x_{z,0}$ is called a \emph{rigid point}; the seminorm $x_{z,0}$ coincides with the seminorm induced by the maximal ideal of $k\{ T \}$ generated by the minimal polynomial of $z$ over $k$. Let $X^{\rig}$ denote the subset of rigid points. The points of $X$ can be classified into 4 types:
\begin{enumerate}[(i)]
\item a type 1 point is of the form $x_{z,0} \in X$ for $z \in K^{\circ}$;
\item a type 2 point is of the form $x_{z,r} \in X$ for $z \in K^{\circ}$ and $r \in (0,1] \cap |K^*|$;
\item a type 3 point is of the form $x_{z,r} \in X$ for $z \in K^{\circ}$ and $r \in (0,1] \backslash |K^*|$;
\item a type 4 point is the pointwise limit of $x_{z_i,r_i}$ such that the corresponding discs $\overline{D}(z_i,r_i)$ form a decreasing sequence with empty intersection.
\end{enumerate}
See~\cite[p.18]{berkovich} for further details; when $K \not= k$, see also~\cite[Proposition 2.2.7]{kedlaya}.
When $k$ is nontrivially-valued, the set of rigid points and the set of type 2 points are both dense in $X$. Points of type 4 exist only when $k$ is not spherically complete. 

For any $x \in X$, the \emph{multiplicity} of $x$ is $m(x) \coloneqq \# \pi^{-1}(x) \in \Z_{>0} \cup \{ \infty \}$. The points of type 1 with finite multiplicity are precisely the rigid points. All points of type 2 and type 3 have finite multiplicity. 

\begin{lemma}\label{multiplicity_descent}
Assume $k$ is trivially-valued or has residue characteristic zero.
For any $x \in X$ of type 2 or 3, there exists $x' \in X^{\rig}$ such that $x' \leq x$ and $m(x') = m(x)$.
\end{lemma}

The statement of~\cref{multiplicity_descent} is not necessarily true for a nontrivially-valued field of positive residue characteristic. For example, if $k = \Q_p$ and $x$ is the type 2 point of $X_{\C_p}$ corresponding to the closed disc $\overline{D}(p^{1/p},|p|^{\frac{2p-1}{p(p-1)}})$ in $(\C_p)^{\circ}$, then $\pi(x) \in X$ is a type 2 point of with $m(\pi(x)) = 1$, but there is no $\Q_p$-rational point lying below it.

\begin{proof}
Let $d = m(x)$, let $x_{1},\ldots,x_{d}$ be the $\pi$-preimages of $x$, and let $G \coloneqq \Gal(k^a/k)$.
Each $x_i$ is the sup-norm over a closed disc $\overline{D}_i \subseteq K^{\circ}$ of radius $r(x) > 0$. 
It suffices to show that there is $z_1 \in \overline{D}_1 \cap (k^a)^{\circ}$ such that $G \cdot \{ z_1 \}  = \{ z_1,\ldots, z_d \}$ with $z_i \in \overline{D}_i$. 
Fix any $z_1 \in \overline{D}_1 \cap (k^a)^{\circ}$, then it is easy to see that any such $z_1$ has $\# G \cdot \{ z_1 \} \geq d$. 

Assume $k$ is trivially-valued. If $r(x) = 1$, the problem is trivial. If $r(x) < 1$, then each $\overline{D}_i$ contains a unique rigid point; in particular, $G \cdot \{ z_1 \}$ has size precisely equal to $d$.

Assume now that $k$ has residue characteristic zero. The case when $d=1$ is~\cite[Proposition 2']{ax-zeros} and we deduce the general case using a similar strategy (the $d=1$ case also follows from the main result of~\cite{schmidt}).

Let $z_1,\ldots,z_{\ell} \in G \cdot \{ z_1 \}$ be those conjugates of $z_1$ that lie in $\overline{D}_1$. We claim each $\overline{D}_i$ contains precisely $\ell$ elements of the orbit $G \cdot \{ z_1 \}$. 
Indeed, by assumption there exists $\sigma_i \in G$ such that $\sigma_i(z_1) \in \overline{D}_i$, from which it follows that $\sigma_i(z_j) \in \overline{D}_i$ for all $j=1,\ldots,\ell$ since $G$ acts by isometries. As $\sigma_i$ is an automorphism of $k^a$, it must give a bijection $G \cdot \{ z_1 \} \cap \overline{D}_1 \stackrel{\sim}{\to} G \cdot \{ z _1 \} \cap \overline{D}_i$, which gives the claim.
Now, for all $i=1,\ldots,d$, set
$$
w_i \coloneqq \ell^{-1} \sum_{u \in G \cdot \{ z_1 \} \cap \overline{D}_i} u.
$$
As the residue characteristic of $k$ is zero,  $|\ell | = 1$ and hence
$$
|\sigma_i(z_1) - w_i| = \left| \ell \sigma_i(z_1) - \sum_{u \in G \cdot \{ z_1 \} \cap \overline{D}_i } u \right| \leq \max_{u \in G \cdot \{ z_1 \} \cap \overline{D}_i } |\sigma_i(z_1) - u| \leq r(x).
$$
In particular, $w_i \in \overline{D}_i$. Moreover, we have $G \cdot \{ w_1 \} = \{ w_1,\ldots, w_d \}$ by construction.
\end{proof}

The Berkovich unit disc $X$ comes equipped with a structure sheaf $\O_X$ on the $G$-topology on $X$. The $k$-algebra of global sections of $\O_X$ is the Tate algebra $k\{ T \}$. See~\cite[2.3]{berkovich} for further details.

\subsection{Quasisubharmonic functions on the Berkovich unit disc}\label{section:quasisubharmonic} 
In this section, we discuss the metric tree structure and potential theory on the Berkovich unit disc. For a comprehensive treatment, see~\cite{baker-rumely} and~\cite{dynberko}.

For each $x \in X$, define an equivalence relation $\sim$ on the set $X \backslash \{ x \}$ by declaring $y \sim z$ if the paths $(x,y]$ and $(x,z]$ intersect. The \emph{tangent space} $T_{X,x}$ at $x$ is the set of equivalences classes of $X \backslash \{ x \}$ modulo $\sim$.
For each $\vec{v} \in T_{X,x}$, let $U(\vec{v})$ be the set of points of $X$ representing $\vec{v}$, and set $m(\vec{v}) \coloneqq \inf_{y \in U(\vec{v})} m(y)$. 
The subsets of the form $U(\vec{v})$, for some tangent direction $\vec{v}$, form a subbasis of open sets for the topology on $X$. See~\cite[\S2.3]{dynberko} for further details.

The closed unit disc $X$ may be equipped with a generalized metric, in the sense that the distance between two points may be infinite; in particular, $X$ gains the structure of a metric tree. This generalized metric is described as follows.  Let $r \colon X \to [0,1]$ denote the radius function, which can be thought of as a $\Gal(k^a/k)$-equivariant function on $X_{K}$.
Define a function $\alpha \colon X \to [0,+\infty]$ by specifying that $\alpha(x_G) = 0$ and 
\begin{equation}\label{definition of alpha}
\alpha(x) - \alpha(y) = -\int_y^x \frac{1}{m(z)} d \left( \log r(z) \right),
\end{equation}
for any two distinct points $x,y \in X$, where the integral is taken over the unique path in $X$ joining the points $x$ and $y$. 
These constraints completely determine a function $\alpha \colon X \to [-\infty, + \infty]$ whose restriction to any segment is monotone decreasing. Observe that, when $k = k^a$, $\alpha = -\log r$. This, in turn, induces a generalized metric $d$ on $X$ by setting
$$
d(x,y) \coloneqq |\alpha(x) - \alpha(x \vee y)| + |\alpha(y) - \alpha(x \vee y)|
$$
for $x, y \in X$; here, $x \vee y$ is the least upper bound of $x$ and $y$.
The rooted tree $(X,\leq )$ acquires the structure of a metric tree
when equipped with the generalized metric $d$. It is important to note that the topology on $X$ induced by $d$ is strictly finer than the native topology.

One can discuss potential theory and the notion of quasisubharmonic functions on any metric tree, as developed in~\cite[\S2.5]{dynberko}. We briefly recall this theory in the special case of the Berkovich disc. This is also discussed in~\cite[\S 5]{baker-rumely} (though our conventions differ slightly). 

Fix a finite atomic measure $\rho_0$ supported at $x_G$, i.e.\ $\rho_0$ is a positive real multiple of the Dirac mass $\delta_{x_G}$ at the Gauss point $x_G$. A function $\varphi \colon X \to [-\infty, \infty)$ is called \emph{$\rho_0$-subharmonic} if it satisfies:
\begin{enumerate}
\item for every finite subtree $Y \subset X \backslash X^{\rig}$ containing $x_G$, $\varphi |_Y$ is a continuous function on $Y$ such that:
\begin{enumerate}
\item $\varphi|_Y$ is convex on any segment in $Y$ that does not contain $x_G$;
\item for any $y \in Y$, $$\rho_0 \{ y \} + \displaystyle\sum_{\vec{v} \in T_y Y} d_{\vec{v}} \left( \varphi|_Y \right) \geq 0,$$ where $d_{\vec{v}} \left( \varphi |_Y \right)$ denotes the directional derivative of $\varphi|_Y$ in the direction $\vec{v}$.
\end{enumerate}
\item $\varphi$ is the limit of its retractions to finite subtrees containing $x_G$; more precisely, if $\{ Y_{i} \}_{i \in I}$ denotes the net of finite subtrees of $X \backslash X^{\rig}$ containing $x_G$ and if $\mathfrak{r}_i \colon X \to Y_i$ is the (continuous) retraction map of $X$ onto $Y_i$, then $\varphi = \ds\lim_i \mathfrak{r}_i^*\varphi$.
\end{enumerate}
The condition (b) is equivalent to the subaverage property: for any $y \in Y$, there exists $r > 0$ such that 
$$
\varphi(y) \leq \frac{1}{|B_Y(y,r)|} \sum_{z \in B_Y(y,r)} \varphi(z) + \mass(\rho_0)r \cdot \mathbf{1}_{\{ x_G \}}(y),
$$
where $B_Y(y,r) \coloneqq \{ z \in Y \colon d(y,z) = r \}$ is the ball of radius $r$ about $y$ in $Y$, and $\mathbf{1}_{\{ x_G\}}$ is the indicator function of the point $x_G$. This is reminiscient of the classical definition of subharmonic functions on $\mathbf{R}^n$. 

A function is said to be \emph{quasisubharmonic} if it is $\rho_0$-subharmonic for some measure $\rho_0$ as above. The class of $\rho_0$-subharmonic functions on $X$ is a convex set, which is closed under taking finite maxima and decreasing (pointwise) limits. A quasisubharmonic function is upper-semicontinuous, but it may take the value $-\infty$; this can only occur at those points $x \in X$ such that $\alpha(x) = +\infty$.

Given a topological space $Z$ and a function $\phi \colon Z \to [-\infty,\infty)$ that is locally bounded above, the \emph{upper-semicontinuous (usc) regularization} $\phi^*$ of $\phi$ is defined by the formula
$$
\phi^*(z) \coloneqq \limsup_{y \to z} \phi(y), \textrm{\hspace{12pt} $z \in Z$.}
$$
The usc regularization $\phi^*$ is the smallest usc function such that $\phi^* \geq \phi$. 

\begin{lemma}\label{lemma:baker_rumely}
If $\{ \varphi_i \}_{i \in I}$ is a net of $\rho_0$-subharmonic functions on $X$ which is locally bounded above, then the usc regularization $\psi^*$ of $\psi \coloneqq \sup_{i \in I} \varphi_i$ is $\rho_0$-subharmonic. Furthermore, $\psi^* = \psi$ on $X \backslash \{ \alpha = +\infty \}$.
\end{lemma}

The proof of \cref{lemma:baker_rumely} is a minor variation on~\cite[Proposition 8.23(E)]{baker-rumely}.

The rest of this section is devoted to a brief discussion of the Laplacian of a quasisubharmonic function. Given a $\rho_0$-subharmonic function $\varphi$ and any finite subtree $Y \subset X$ containing $x_G$, let $\Delta (\varphi|_{Y})$ be the unique signed Borel measure on $Y$ determined by the following rule: 
if $\vec{v}_1,\ldots,\vec{v}_n$ are tangent directions in $Y$ such that $U(\vec{v}_i) \cap U(\vec{v}_j) \not= \emptyset$ and $U(\vec{v}_i) \not\subseteq U(\vec{v}_j)$ for all $i \not= j$, then
$$
\Delta (\varphi|_Y) \left( \bigcap_{i=1}^n U(\vec{v_i}) \right) = -\sum_{i=1}^n d_{\vec{v_i}} \left( \varphi|_Y \right).
$$
The \emph{Laplacian} is the signed Borel measure $\Delta \varphi$ uniquely characterized by the following property: for any finite subtree $Y \subset X$ containing $x_G$,
$$
(\mathfrak{r}_Y)_*(\rho_0 + \Delta \varphi) = \rho_0 + \Delta \left( \varphi|_Y \right),
$$
where $\mathfrak{r}_Y \colon X \to Y$ is the (continuous) retraction map of $X$ onto $Y$.
There are several Laplacians defined in the literature, and the above definition follows the conventions of~\cite{favre-jonsson,dynberko} and it differs from that of~\cite{baker-rumely} by a negative sign; see~\cite[\S 5.8]{baker-rumely} for a discussion of the various Laplacian operators.

Quasisubharmonic functions and their Laplacians behave well with respect to the ground field extension morphism $\pi \colon X_K \to X$, in the following sense: given a $\rho_0$-subharmonic function $\varphi$ on $X$, $\pi^* \varphi$ is a $\pi^*\rho_0$-subharmonic function on $X_K$ and $\Delta \varphi = \pi_* \Delta (\pi^* \varphi)$.

\begin{example}
For any irreducible $f \in k\{ T \}$, the function $\varphi := \log |f|$ is $\rho_0$-subharmonic, where $\rho_0 = m(x)\cdot \delta_{x_G}$ and $x \in X^{\rig}$ is the rigid point of $X$ corresponding to the maximal ideal $(f)$ of $k\{ T \}$.
 It is easy to check that 
$$
\Delta \varphi =  m(x) \left( \delta_{x} - \delta_{x_G} \right).
$$
More generally, for any $f \in k\{ T \}$, the function $\log |f|$ is quasisubharmonic and its Laplacian can be identified, up to scaling and adding a multiple of $\delta_{x_G}$, with the divisor of zeros of $f$ via the Poincar\'e--Lelong formula. See~\cite[Example 5.20]{baker-rumely}.

In fact, for any $x \in X$ with $m(x) < +\infty$, the function $\varphi = -\alpha(x \vee \cdot)$ is $m(x) \cdot \delta_{x_G}$-subharmonic with Laplacian given by $\Delta \varphi = m(x) \left( \delta_{x}-\delta_{x_G} \right)$.
\end{example}

For a quasisubharmonic function $\varphi$ on $X$, we may construct a metric on $\O_X$: to a local section $f$ of $\O_X$ over an analytic domain $V \subset X$, we assign the function $x \mapsto |f(x)|e^{-\varphi(x)}$, for $x \in V$. This convention mirrors how (semipositive) metrics on line bundles on complex manifolds are locally given by plurisubharmonic functions, as in e.g.~\cite[\S 3]{demailly12}.

\subsection{Temkin's metric}\label{section:temkin} In~\cite{temkin}, Temkin introduces a canonical lower-semicontinuous metric $A_Z$ on the canonical bundle $\omega_{Z/k}$ of any $k$-analytic space $Z$. We will explain the construction of $A_Z$ when $Z$ is a $k$-affinoid space. In the sequel, the metric $A_Z$ will play a crucial role when $Z$ is the closed unit disc.

To a contractive morphism $\mathcal{A} \to \mathcal{B}$ of seminormed rings, the module $\Omega_{\mathcal{B}/\mathcal{A}}$ of K\"ahler differentials may be equipped with the maximal seminorm $\| \cdot \|_{\mathcal{B}/\mathcal{A}}$ making the derivation $d \colon \mathcal{B} \to \Omega_{\mathcal{B}/\mathcal{A}}$ into a contractive $\mathcal{A}$-module morphism. Explicitly,
$$
\| s \|_{\mathcal{B}/\mathcal{A}} = \inf_{s = \sum_i b_i d(b_i')} \max_i |b_i| \cdot |b_i'|,
$$
where the infimum is taken over all representations $s = \sum_i b_i d(b_i')$. 
Let $\widehat{\Omega}_{\mathcal{B}/\mathcal{A}}$ denote the separated completion of $\Omega_{\mathcal{B}/\mathcal{A}}$ with respect to $\| \cdot \|_{\mathcal{B}/\mathcal{A}}$; this is a Banach $\mathcal{B}$-module. We say that $\widehat{\Omega}_{\mathcal{B}/\mathcal{A}}$ is the \emph{completed module of K\"ahler differentials} and we say its norm $\| \cdot \|_{\mathcal{B}/\mathcal{A}}$ is the \emph{K\"ahler norm} on $\widehat{\Omega}_{\mathcal{B}/\mathcal{A}}$. See~\cite[4.1]{temkin} for further details.

For a $k$-affinoid algebra $\mathcal{A}$, the module $\widehat{\Omega}_{\mathcal{A}/k}$ has an alternate description: if $\mathcal{J}$ is the kernel of the multiplication map $\mathcal{A} \hotimes_k \mathcal{A} \to \mathcal{A}$, then $\widehat{\Omega}_{\mathcal{A}/k} = \mathcal{J}/\mathcal{J}^2$. The derivation $d \colon \mathcal{A} \to \widehat{\Omega}_{\mathcal{A}/k}$ is induced from the bounded map $\mathcal{A} \to \mathcal{J}$ of Banach $\mathcal{A}$-modules given by $a \mapsto 1 \hotimes a - a \hotimes 1$. The image of  the derivation $d$ generates $\widehat{\Omega}_{\mathcal{A}/k}$ as an $\mathcal{A}$-module. This is the approach taken in~\cite[3.3]{berkovich93}.

Let $Z = \M(\mathcal{A})$ be a $k$-affinoid space. The finite Banach $\mathcal{A}$-module $\widehat{\Omega}_{\mathcal{A}/k}$ determines a coherent $\O_Z$-module $\Omega_{Z/k} \coloneqq \O_Z(\widehat{\Omega}_{\mathcal{A}/k})$ on $Z$. Temkin provides a canonical metric on each stalk of $\Omega_{Z/k}$, as well as on the stalks of its exterior powers (and in particular, on the canonical bundle). This was originally introduced as a generalization of the weight function of~\cite{mustata-nicaise}. It has arisen in other contexts more recently, e.g.\ as a log discrepancy-type function in~\cite[5.6]{boucksom-jonsson}. 

For any $z \in Z$, there is a natural homomorphism of $k$-algebras
$$
\psi_z \colon \Omega_{Z/k,z} \longrightarrow \Omega_{Z/k,z} \otimes_{\O_{Z,z}} \H(z) \stackrel{\simeq}{\longrightarrow} \widehat{\Omega}_{\H(z)/k},
$$
which extends the natural isomorphism from the completed fibre of $\Omega_{Z/k}$ at $z$ to the completed module of K\"ahler differentials of $\H(z)/k$. Let $\pi \colon Z_K \to Z$ denote the ground field extension of $Z$. 

Fix $z \in Z$. The canonical metric $\| \cdot \|_z$ on $\Omega_{Z/k,z}$ is defined by the following rule: for any germ $s \in \Omega_{Z/k,z}$, set
$$
\| s \|_z \coloneqq \| \psi_{z'} \left( \pi^* s \right) \|_{\H(z')/k}
$$
for any $z' \in \pi^{-1}(z)$. This is independent of the choice of $z'$. In particular, if $k = k^a$, then $\| s \|_z = \| \psi_z(s) \|_{\H(z)/k}$, i.e. $\| \cdot \|_z$ is the norm on $\Omega_{Z/k,z}$ defined by pulling back the norm on $\widehat{\Omega}_{\H(z)/k}$. This metric naturally extends to the the exterior powers of $\Omega_{Z/k}$ and, in particular, to the canonical bundle $\omega_{Z/k}$. For the various properties of the canonical metric, see~\cite[\S6]{temkin}.

The canonical metric, as defined above, is called the ``geometric K\"ahler norm" in~\cite[6.3.15]{temkin}. When $k \not= k^a$, if one defines the metric on the stalk $\Omega_{Z/k,z}$ by pulling back the metric on $\widehat{\Omega}_{\H(z)/k}$ (without first passing to the completed algebraic closure), then the metric may be quite poorly behaved (for example when $k$ has wild extensions). See~\cite[6.2]{temkin}.

If $\omega_{Z/k}$ is trivial, a trivializing section $\eta$ of $\omega_{Z/k}$ gives rise to a lower-semicontinuous function $A_Z \colon Z \to [0,+\infty]$ defined by the following property: for any $z \in Z$, write a germ $s \in \omega_{Z/k,x}$ as $s = f \cdot \eta_x$ for some $f \in \O_{Z,z}$, where $\eta_x$ is the image of $\eta$ in $\omega_{Z/k,z}$. Then,
$$
\| s \|_z = |f(z)|e^{-A_Z(z)}.
$$
The function $A_Z$ depends on the choice of trivializing section $\eta$, but we suppress this dependence in the notation.

For our purposes, the important special case is when $X = \M(k\{ T \})$ is the closed unit disc. The completed module $\widehat{\Omega}_{k\{ T \}/k}$ of K\"ahler differentials is the free Banach $k\{ T \}$-module on the differential $dT$ and, following~\cite[6.2.1]{temkin}, the canonical metric admits a simple description: for any $x \in X$ and any global section $f \cdot dT$ of $\widehat{\Omega}_{k\{T\}/k}$, we have
$$
\| f \cdot dT \|_x = |f(x)| r(x) = |f(x)| e^{-A_X(x)},
$$
where $r \colon X \to [0,1]$ is the radius function. Thus, $A_X = -\log r$. This description holds regardless of whether or not $k$ is algebraically closed. Observe that $A_X = +\infty$ on the type 1 points of $X$.

The function $A_X$ is convex on any segment of $X$ and, in fact, its directional derivatives admit a simple expression: if $x \in X$ with $m(x) < +\infty$ and $\vec{v} \in T_{X,x}$ is a direction such that $U(\vec{v}) \cap X^{\rig} \not= \emptyset$, then $d_{\vec{v}} A = m(\vec{v})$ .
Moreover,~\cref{definition of alpha} shows that $\frac{A_X(x)}{m(x)} \leq \alpha(x) \leq A_X(x)$ for $x \in X$.

In~\S\ref{section:ot} and~\S\ref{section:application}, we will be concerned with ``twists'' of the canonical metric $A_X$ on $\Omega_{X/k}$ by a metric on $\O_X$. More precisely, for any quasisubharmonic function $\varphi$ on $X$ and any $x \in X$, we consider the metric $\| \cdot \|_{\varphi, x}$ which, to a global section $f \cdot dT$ of $\Omega_{X/k}$, assigns the number
$$
\| f \cdot dT \|_{\varphi,x} \coloneqq |f(x)|e^{-\varphi(x) - A_X(x)}.
$$
In fact, we will be interested in $\sup_{x \in X} \| f \cdot dT \|_{\varphi,x}$, which we write as $\| f \|_{\varphi}$ in the sequel.

\section{An Ohsawa--Takegoshi-type extension theorem}\label{section:ot}

Let us briefly recall the notation established in~\S\ref{section:preliminaries}. Let $k$ be a complete non-Archimedean field, and assume in addition that $k$ is algebraically closed, trivially-valued, or is spherically complete of residue characteristic zero.
Let $X = \M(k\{ T \})$ be the Berkovich closed unit disc over $k$ with Gauss point $x_G$, let $r \colon X \to [0,1]$ be the radius function, and let $A \coloneqq -\log r \colon X \to [0,+\infty]$ be Temkin's metric.

In this section, we prove the following variant of~\cref{thm:main} and derive~\cref{thm:main} as an easy consequence.

\begin{theorem}\label{thm:otvariant}
Let $\varphi$ be a quasisubharmonic function on $X$ with $\varphi(x_G) = 0$, and let $Z(\varphi) = \{ \varphi = -\infty \}$ denote the polar locus of $\varphi$. For any $z \in X$, there exists a constant $\epsilon_0 > 0$ and a nonzero polynomial $f \in k[T]$ such that
$$
\| f \|_{(1+\epsilon)\varphi} \coloneqq \sup_{x \in X \backslash Z(\varphi)} |f(x)|e^{-(1+\epsilon)\varphi(x) - A(x)} \leq |f(z)|e^{-\varphi(z)} \textrm{ for all $\epsilon \in [0,\epsilon_0]$.}
$$
If $\varphi(z) = -\infty$, then we may find $f$ such that $\| f \|_{(1+\epsilon)\varphi} < +\infty$ for all $\epsilon \in [0,\epsilon_0]$. Moreover, if $k$ is algebraically closed and $z \in X^{\rig}$, then for any value $a \in \H(z)^* = k^*$, we may find $f$ such that $f(z) = a$. 
\end{theorem}

We will prove~\cref{thm:otvariant} in~\S\ref{section:proof_variant}. 
The hypotheses of~\cref{thm:otvariant} may be weakened to allow $\varphi(x_G) \geq 0$, but it is false if $\varphi(x_G) < 0$ (e.g.\ if $\varphi = -1$, then the only $f$ that could satisfy the inequality in~\cref{thm:otvariant} at the Gauss point is $f = 0$).
Nonetheless,~\cref{thm:main}, which has no hypothesis on the value of $\varphi(x_G)$, may be easily deduced from~\cref{thm:otvariant}.
Furthermore, the presence of the hypotheses on the field $k$ is discussed immediately after the proof of~\cref{lemma:three_variant}

\begin{proof}[Proof of~\cref{thm:main}]
Given a quasisubharmonic function $\varphi$ on $X$ and a point $z \in X$, set $\phi \coloneqq \varphi - \varphi(x_G)$.~\cref{thm:otvariant} asserts that there is a nonzero polynomial $f \in k[T]$ such that $\| f \|_{(1+\epsilon)\phi } \leq |f(z)| e^{-\phi(z)}$ for all $\epsilon > 0$ sufficiently small. Thus,
$$
\lim_{\epsilon \to 0^+} \| f \|_{(1+\epsilon)\varphi } = e^{\varphi(x_G)} \lim_{\epsilon \to 0^+} e^{\epsilon \varphi(x_G)} \| f \|_{(1+\epsilon)\phi } \leq |f(z)| e^{-\phi(z) + \varphi(x_G)} = |f(z)|e^{-\varphi(z)}.
$$
This completes the proof of~\cref{thm:main}.
\end{proof}

As discussed in~\S\ref{section:intro}, the optimal constant appearing in the classical Ohsawa--Takegoshi theorem for the complex unit disc is $\pi$. In the non-Archimedean setting, however, the optimal constant is $1$, in the following sense.

\begin{corollary}
For any $z \in X$, let $c(z)$ be the smallest positive number such that for any quasisubharmonic function $\varphi$ on $X$, there exists a nonzero $f \in k\{ T \}$ satisfying
$$
\lim_{\epsilon \to 0^+} \| f \|_{(1+\epsilon)\varphi } \leq c(z)|f(z)|e^{-\varphi(z)}.
$$
Then, $c(z) = 1$. 
\end{corollary}

\begin{proof}
If $\varphi \equiv 0$, then $|f(x_G)| = \lim_{\epsilon \to 0^+} \| f\|_{(1+\epsilon)\varphi } \leq c(z) |f(z)| \leq c(z) |f(x_G)|$. As $|f(x_G)| \not= 0$, it follows that $c(z) \geq 1$.~\cref{thm:main} asserts that the lower bound $c(z) = 1$ is achieved for any $z$.
\end{proof}

\subsection{Proof of~\texorpdfstring{\cref{thm:otvariant}}{Theorem \ref{thm:otvariant}}}\label{section:proof_variant}

Fix $z \in X$.
Suppose $\varphi$ is $\rho_0$-subharmonic for some finite (atomic) measure $\rho_0$ supported at the Gauss point $x_G$, and $\varphi(x_G) = 0$. To simplify the exposition, we assume $\varphi(z) > -\infty$. When $\varphi(z) = -\infty$, the proof is similar: one proves analogues of the sequence of lemmas below, each of which is made easier because one is only concerned with ensuring that a quantity is finite, as opposed to ensuring that the same quantity is less than some fixed value. 

The strategy of the proof is to use strong induction on $\lfloor \mass(\rho_0) \rfloor$. For this reason, the following terminology will be quite helpful.

\begin{definition} 
Let $\epsilon_0 > 0$. A nonzero polynomial $f \in k[T]$ is a \emph{$(\varphi,\epsilon_0)$-extension at $z$} if the inequality 
$$\| f \|_{(1+\epsilon)\varphi } \leq |f(z)|e^{-\varphi(z)}$$
holds for all $\epsilon \in [0,\epsilon_0]$. Equivalently, $f$ is a $(\varphi,\epsilon_0)$-extension if for all $x \in X \backslash Z(\varphi)$ and all $\epsilon \in [0,\epsilon_0]$, we have
\begin{equation}
\label{extension inequality}
\log |f(x)| - \log |f(z)| \leq A(x) + (1+\epsilon)\varphi(x) - \varphi(z).
\end{equation}
\end{definition}
As $\varphi \leq 0$, in order to verify that $f$ is an $(\varphi,\epsilon_0)$-extension at $z$, it suffices to check $\| f \|_{(1+\epsilon_0)\varphi} \leq |f(z)|e^{-\varphi(z)}$, as opposed to verifying it for all $\epsilon \in [0,\epsilon_0]$.

The key players in the proof of~\cref{thm:otvariant} are the finite subtrees
$$
\Gamma_{\varphi,n} \coloneqq \left\{ x \in X \colon (\rho_0 + \Delta \varphi) \{ y \leq x \} \geq \frac{n}{n+1} m(x) \right\}.
$$
for $n \geq 1$. 
Note that $\Gamma_{\varphi,n}$ is indeed a finite subtree of $X$, since any end $x$ of $\Gamma_{\varphi,n}$ satisfies $(\rho_0 + \Delta \varphi) \{ x \} \geq \frac{n}{n+1} m(x)$ and $\rho_0 + \Delta\varphi$ is a positive measure of finite mass equal to $\textrm{mass}(\rho_0)$.
It is clear that $\Gamma_{\varphi,n} \supseteq \Gamma_{\varphi,n+1}$ for any $n \geq 1$; in particular,
$$
\Gamma_{\varphi} \coloneqq \bigcap_{n \geq 1} \Gamma_{\varphi,n} = \left\{ x \in X \colon (\rho_0 + \Delta \varphi) \{ y \leq x \} \geq m(x) \right\}
$$
is a finite subtree of $X$ that is contained in $\Gamma_{\varphi,n}$ for any $n \geq 1$.

The subtrees $\Gamma_{\varphi,n}$ are crucial in reducing~\cref{thm:otvariant} to a ``finite'' problem.
Observe that any point of $\Gamma_{\varphi,n}$ must have a finite multiplicity because $\mass(\rho_0)$ is finite and, moreover, $\Gamma_{\varphi,n} \cap X^{\rig} \subseteq Z(\varphi)$. Variants of this tree appear in \cite[Prop 2.8]{dynberko} and~\cite[Lem 7.7]{favre-jonsson}.

The following is the base case for our induction.

\begin{lemma}\label{lemma:one_variant}
If $\mass(\rho_0) < 1$, there exists a constant $\epsilon_0 > 0$ such that any nonzero constant function is a $(\varphi,\epsilon_0)$-extension at $z$.
\end{lemma}

\begin{proof}
Set $\epsilon_0 \coloneqq \frac{1-\mass(\rho_0)}{\mass(\rho_0)}$. Since $\mass(\rho_0) < 1$, $\Gamma_{\varphi} = \emptyset$ and $|d_{\vec{v}} (1+\epsilon) \varphi| \leq (1+\epsilon)\mass(\rho_0) \leq 1$ for all tangent directions $\vec{v}$ in $X$, provided $\epsilon \in [0,\epsilon_0]$. To check that any constant function is a $(\varphi,\epsilon_0)$-extension at $z$, it suffices to check that $0 \leq A + (1+\epsilon)\varphi - \varphi(z)$ on $X \backslash Z(\varphi)$. This inequality is satisfied at the Gauss point:
$$
A(x_G) + (1+\epsilon)\varphi(x_G) - \varphi(z) = -\varphi(z) \geq 0.
$$
Moreover, the convexity of $\varphi$ and of $A$ ensures that it is enough to check that, in any tangent direction at $x_G$, the function $(1+\epsilon)\varphi + A$ only increases in that direction. 
However, in any such direction $\vec{v}$, 
$$d_{\vec{v}}\left( (1+\epsilon)\varphi + A \right) \geq -(1+\epsilon)\mass(\rho_0) + 1 \geq 0,$$
which completes the proof.
\end{proof}

Now, if $\mass(\rho_0) \geq 1$, then $x_G \in \Gamma_{\varphi,n}$ for all $n \geq 1$, because $m(x_G) = 1$. In particular, $\Gamma_{\varphi,n} \not= \emptyset$ and the set of ends\footnote{If $\Gamma$ is a subtree of $X$ that contains $x_G$, we adopt the following conventions: 
if $\Gamma \supsetneq \{ x_G \}$, then $\Ends(\Gamma)$ consists of those points in $\Gamma$ with a unique tangent direction in $\Gamma$;
if $\Gamma = \{ x_G \}$, then $\Ends(\Gamma) = \Gamma$.
} $\Ends(\Gamma_{\varphi,n})$ is nonempty. 
Let $\Gamma'_{\varphi,n}$ denote the convex hull of $\Gamma_{\varphi,n} \cup \{ z \}$. Let $\mathfrak{r}_{\Gamma'_{\varphi,n}} \colon X \to \Gamma'_{\varphi,n}$ denote the retraction of $X$ onto $\Gamma'_{\varphi,n}$. 
For any $n \geq 1$, observe that a $(\mathfrak{r}_{\Gamma'_{\varphi,n}}^* \varphi, \epsilon_0)$-extension $f$ at $z$ is also a $(\varphi,\epsilon_1)$-extension at $z$, where $\epsilon_1 = \min\{ \epsilon_0, \frac{1}{n} \}$. 
Indeed, for any $x \in \Gamma_{\varphi,n}'$ and any direction $\vec{v} \in T_{X,x}$ such that $U(\vec{v}) \cap \Gamma'_{\varphi,n} = \emptyset$, we have
$$
d_{\vec{v}}( (1 + \epsilon)\varphi + A ) > \left( - (1+\epsilon)\frac{n}{n+1} + 1 \right) m(x) \geq 0
$$
for all $\epsilon \in [0,\epsilon_1]$. Thus, for fixed $n \geq 1$, we may replace $\varphi$ with $\mathfrak{r}_{\Gamma'_{\varphi,n}}^* \varphi$ to assume that $\varphi$ is locally constant off of $\Gamma_{\varphi,n}'$.

Fix $n \geq 1$ such that 
\begin{equation}\label{tree_decomposition}
\Gamma_{\varphi,n} \backslash \Gamma_{\varphi} = \bigsqcup_{j \in J} [x_j, \tilde{x}_j),
\end{equation}
where $J$ is a finite index set, and for all $j \in J$, $\tilde{x}_j \in \Gamma_{\varphi}$ and the point $x_j$ is of type 2 or 3 such that the multiplicity function is constant on $[x_j,\tilde{x}_j)$. Moreover, for each $j \in J$, if $\vec{v}_j \in T_{X,\tilde{x}_j}$ denotes the unique tangent direction at $\tilde{x}_j$ with $x_j \in U(\vec{v}_j)$, then $d_{\vec{v}_j} \varphi = -m(x_j)$.

As mentioned before, the strategy of the proof of \cref{thm:otvariant} is by strong induction on $\lfloor \mass(\rho_0) \rfloor$. In the following sequence of lemmas, we explain how to reduce the problem of the existence of a $(\varphi, \epsilon_0)$-extension at $z$ to the existence of a $(\phi, \epsilon_1)$-extension at $z$, where $\phi$ is a $\rho$-subharmonic for some finite measure $\rho$ with $\mass(\rho) \leq \mass(\rho_0) - 1$ and $0 < \epsilon_1 \leq \epsilon_0$. 
As $\mass(\rho_0)$ is finite, after finitely-many such reductions we must find ourselves in the setting of \cref{lemma:one_variant}. The hypotheses of each lemma are concerned with which types of points may arise in the finite set $\Ends(\Gamma_{\varphi,n})$; in particular, if one assumes that $k$ is spherically complete, one can ignore~\cref{lemma:four_variant}. 

\begin{lemma}\label{lemma:two_variant}
Suppose that $\mass(\rho_0) \geq 1$ and that there exists $x \in \Ends(\Gamma_{\varphi,n})$ of type 1. Then, there exists $\epsilon_1 > 0$ and a $\rho$-subharmonic function $\phi$, with $\mass(\rho) \leq \mass(\rho_0) - 1$, such that a $(\varphi,\epsilon_0)$-extension $f$ at $z$ may be constructed from a $(\phi,\epsilon_1)$-extension $\tilde{f}$ at $z$.
\end{lemma}

\begin{proof}
By assumption, $x \in \Gamma_{\varphi}$ and so it satisfies $m(x) < +\infty$; in particular, $x \in X^{\rig}$, because $x$ is of type 1. 
In addition, as $\Gamma_{\varphi} \cap X^{\rig} \subseteq Z(\varphi)$ and $z \not\in Z(\varphi)$, we have $x \not= z$.
Let $\frakm_x$ denote the maximal ideal of $k\{ T \}$ corresponding to $x$, and let $g$ be a polynomial generator of $\frakm_x$ with $|g(x_G) | = 1$. As type-1 points are minimal with respect to the partial order $\leq$, we have
$$
1 \leq m(x) \leq (\rho_0 + \Delta \varphi) \{ y \in X \colon y \leq x \} = \Delta \varphi \{ x \}.
$$
Set $c \coloneqq \frac{\Delta\varphi\{ x \}}{m(x)} \geq 1$, $\gamma \coloneqq \lfloor c \rfloor$, and $\rho \coloneqq \rho_0 - \gamma\delta_{x_G}$, then the function $\phi \coloneqq \varphi - \gamma\log |g|$ is $\rho$-subharmonic and $\phi(x_G) = 0$. Suppose there exists $\epsilon_1 > 0$ and a $(\phi,\epsilon_1)$-extension $\tilde{f} \in k[T]$ at $z$ and set $f \coloneqq g^{\gamma} \tilde{f}$. We claim that there exists $\epsilon_0 \in (0,\epsilon_1]$ such that $f$ is a $(\varphi,\epsilon_0)$-extension at $z$. 

\emph{Case 1}. Consider the case when $c$ is an integer. If $z' \coloneqq \mathfrak{r}_{\Gamma_{\varphi,n}}(z)$, then $\Gamma_{\varphi,n}' = \Gamma_{\varphi,n} \cup [z,z']$.
If $z' \not= z$, let $\vec{v}_z \in T_{X,z'}$ be the unique direction at $z'$ with $z \in U(\vec{v}_z)$.
The function $y \mapsto |f(y)|e^{-(1+\epsilon_0)\varphi(y)-A(y)}$ is decreasing in the direction $\vec{v}_z$ provided $d_{\vec{v}_z}( - (1+\epsilon_0) \varphi - A) \leq 0$. 
This occurs if $\epsilon_0 \leq \frac{m(\vec{v}_z)}{n m(z')}$. Thus, it suffices to find $\epsilon_0 > 0$ such that 
$$
|f(y)|e^{-(1+\epsilon_0)\varphi(y)-A(y)} \leq |f(z)|e^{-\varphi(z)} = |\tilde{f}(z)|e^{-\phi(z)}
$$
for all $y \in \Gamma_{\varphi,n}$. 
Set $\widetilde{\Gamma}_{\phi,n} = \Gamma_{\phi,n} \cup \{ x_G \}$ (where we must include $x_G$ in case $\Gamma_{\phi,n}$ is empty).
If $y \in \widetilde{\Gamma}_{\phi,n}$, then $\phi(y) \leq -\frac{n}{n+1} \alpha(y)$, and hence$$
|g(y)|^{-\epsilon_0 c} e^{(\epsilon_1 -\epsilon_0)\phi(y)} \leq \left( e^{-\alpha(y)} \right)^{-\epsilon_0 c m(x) + \frac{n}{n+1}(\epsilon_1-\epsilon_0)}
$$
and this is less than or equal to $1$ provided $\epsilon_0 \leq \frac{\epsilon_1 n}{m(x)c(n+1)+n}$.
Granted this, observe that
$$
|f(y)|e^{-(1+\epsilon_0)\varphi(y)-A(y)} = |g(y)|^{-\epsilon_0 c} e^{(\epsilon_1 - \epsilon_0)\phi(y)} |\tilde{f}(y)|e^{-(1+\epsilon_1)\phi(y) - A(y)} \leq |\tilde{f}(y)|e^{-(1+\epsilon_1)\phi(y) - A(y)} \leq |\tilde{f}(z)|e^{-\phi(z)},
$$
as required. Let $\tilde{x}$ denote the minimal element of $\widetilde{\Gamma}_{\phi,n} \cap [x,x_G]$ and let $\vec{v}_x \in T_{X,\tilde{x}}$ be the unique direction at $\tilde{x}$ such that $x \in U(\vec{v}_x)$. 
As $\Gamma_{\varphi,n} = \widetilde{\Gamma}_{\phi,n} \cup [x,x_G]$, it suffices to find an $\epsilon_0 > 0$ such that the function $y \mapsto |f(y)|e^{-(1+\epsilon_0)\varphi(y) - A(y)}$ is decreasing in the direction $\vec{v}_x$, or equivalently
$$
d_{\vec{v}_x}\left( \epsilon_0 m(x) c\alpha -(1+\epsilon_0)\phi - A \right) \leq 0.
$$
This occurs if $\epsilon_0 \leq \frac{m(\vec{v}_x)}{(c(n+1)+n)m(x)}$.

\emph{Case 2}. Now, suppose $c$ is not an integer, i.e.\ $c > \gamma$.
Write $\varphi = \psi + c \log |g|$ for some quasisubharmonic function $\psi$ on $X$ with $\psi \leq 0$. 
Set $\epsilon_0 = \frac{\epsilon_1 (c-\gamma)}{c}$. 
It is straightforward to check that $\epsilon_1 \phi \leq \epsilon_0 \varphi$ on all of $X$.
Thus, for any $y \in X \backslash Z(\varphi)$, we have
$$
|f(y)|e^{-(1+\epsilon_0)\varphi(y)-A(y)} 
= e^{\epsilon_1 \phi(y) - \epsilon_0 \varphi(y)} |\tilde{f}(y)|e^{-(1+\epsilon_1)\phi(y) - A(y)} 
\leq |\tilde{f}(z)|e^{-\phi(z)} = |f(z)|e^{-\varphi(z)},
$$
which completes the proof.

\end{proof}

%

%
%
%

\begin{lemma}\label{lemma:three_variant}
Suppose that $\mass(\rho_0) \geq 1$ and that there exists $x \in \Ends(\Gamma_{\varphi,n})$ of type 2 or of type 3 such that $\mathfrak{r}_{\Gamma_{\varphi,n}}(z) \not= x$.
Then, there exists $\epsilon_1 > 0$ and a $\rho$-subharmonic function $\phi$, with $\mass(\rho) \leq \mass(\rho_0) - 1$, such that a $(\varphi,\epsilon_0)$-extension $f$ at $z$ may be constructed from a $(\phi,\epsilon_1)$-extension $\tilde{f}$ at $z$.
\end{lemma}

\begin{proof}
By~\cref{multiplicity_descent}, there exists $x' \in X^{\rig}$ such that $\mathfrak{r}_{\Gamma_{\varphi,n}}(x') = x$ and $m(x') = m(x)$.
If $x \in \Gamma_{\varphi}$, then set 
$c \coloneqq \frac{\Delta \varphi \{ x \}}{m(x)} \geq 1$.
We construct a new $\rho_0$-subharmonic function $\varphi'$ by extending $\varphi$ linearly with slope $c m(x)$ from the end $x$ to the rigid point $x'$; more precisely, $\varphi'$ is given by the formula
$$
\varphi'(y) := \varphi(y) 
+c m(x)\left( \alpha(y \vee x) - \alpha (y \vee x') \right).
$$
It is clear that $\varphi'(x_G) = 0$, $\varphi'(z) = \varphi(z)$, and $\varphi' \leq \varphi$. In particular, any $(\varphi',\epsilon_0)$-extension at $z$ is also a $(\varphi,\epsilon_0)$-extension at $z$. 
%

If $x \in \Gamma_{\varphi,n} \backslash \Gamma_{\varphi}$, then by the assumption~\cref{tree_decomposition}, we can 
replace $\varphi$ with a function that is linear on $[x, \mathfrak{r}_{\Gamma_{\varphi}}(x)]$ with slope $m(\mathfrak{r}_{\Gamma_{\varphi}}(x))$ and then repeat the same argument as above.
\end{proof}

The proof of~\cref{lemma:three_variant} is one point in the proof of~\cref{thm:otvariant} where the additional assumptions on the field $k$ are needed (in order to be able to apply~\cref{multiplicity_descent}). If one were to pick $x' \in X^{\rig}$ such that $m(x') > m(x)$, then the function $\varphi'$ need not be quasisubharmonic. 

In addition, the conditions on the field $k$ are such that the hypotheses of the following lemma can only occur when $k$ is algebraically closed and nontrivially-valued.

\begin{lemma}\label{lemma:four_variant}
Suppose that $\mass(\rho_0) \geq 1$ and that there exists $x \in \Ends(\Gamma_{\varphi,n})$ of type 4 and $x \not= z$. Then, there exists $\epsilon_1 > 0$ and a $\rho$-subharmonic function $\phi$, with $\mass(\rho) \leq \mass(\rho_0) - 1$, such that a $(\varphi,\epsilon_0)$-extension $f$ at $z$ may be constructed from a $(\phi,\epsilon_1)$-extension $\tilde{f}$ at $z$.
\end{lemma}

\begin{figure}[ht!]
\centering
  \begin{tikzpicture}[yscale=2.5]
    \draw[thick] (0,0) -- (0,1) node[left,xshift=-3pt] {$x_G$};
    \draw[thick] (-1,-0.5) node[left,xshift=-3pt] {$r_{\Gamma_{\varphi}}(z)$} -- (0,0) node[left,xshift=-3pt] {$z \vee x$};
    \draw[thick, dashed] (-1,-0.5) -- (-2,-1) node[left,xshift=-3pt] {$z$};
    \draw[thick] (0,0) -- (5, -5/7) node[right,xshift=3pt] {$x$};
    \draw[thick] (3,-3/7) node[right,yshift=5pt] {$\tilde{x}$} -- (2.5,-1) node[right,xshift=3pt] {$$};
    \draw[thick, dashed] (4.5,-9/14) node[right,yshift=5pt] {$u$} -- (4.4,-1) node[right,xshift=3pt] {$u'$};
    \draw[thick] (2,-2/7) -- (11/8,-1) node[left,xshift=-3pt] {$$};
    \node at (0,1) {\large $\bullet$};
    \node at (0,0) {\large $\bullet$};
    \node at (-1,-0.5) {\large $\bullet$};
    \node at (-2,-1) {\large $\bullet$};
    \node at (4.5,-9/14) {\large $\bullet$};
    \node at (5,-5/7) {\large $\bullet$};
    \node at (2.5,-1) {\large $\bullet$};
    \node at (3,-3/7) {\large $\bullet$};
    \node at (2,-2/7) {\large $\bullet$};
    \node at (11/8,-1) {\large $\bullet$};
    \node at (4.4,-1) {\large $\bullet$};
  \end{tikzpicture}
\caption{One possible configuration for $\Gamma_{\varphi,n}$ in the setting of~\cref{lemma:four_variant}.}\label{figure:lemma_four_variant}
\end{figure}
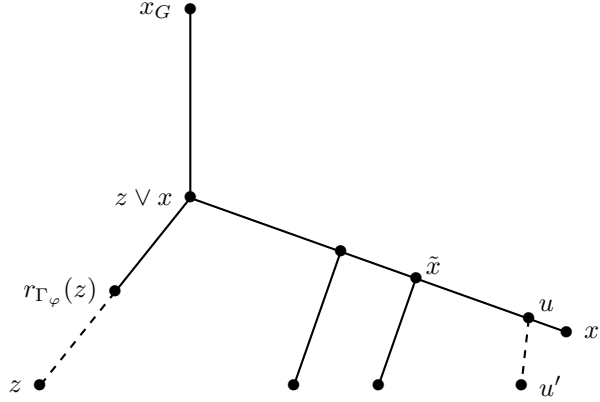

\begin{proof}
By assumption, $k$ is algebraically closed and nontrivially-valued (so the type 2 points of $X$ are dense). 
Moreover, all multiplicities are equal to 1 and $\alpha = A$. 

Note that, since $x$ is of type 4, $A(x) \in (0,+\infty)$.
Recall that $\Gamma'_{\varphi,n}$ denotes the convex hull of $\Gamma_{\varphi,n} \cup \{ z \}$ (we allow the possibility that $z \in \Gamma_{\varphi,n}$, so it is possible that $\Gamma'_{\varphi,n} = \Gamma_{\varphi,n}$). Let $\tilde{x} \in (x,x_G]$ be the minimal point with the property that $\{ y \in \Gamma'_{\varphi,n} \colon y \leq \tilde{x} \} \not= [x,\tilde{x}]$. That is, $\tilde{x}$ is the minimal point on $(x,x_G]$ such that $(\tilde{x},x]$ does not intersect any branches of $\Gamma'_{\varphi,n}$ other than the one containing $x$. 
It follows that $\mathfrak{r}_{\Gamma_{\varphi,n}}(z) \not\in (\tilde{x},x]$. 
An example is in~\cref{figure:lemma_four_variant}.

Set $c \coloneqq-d_{\vec{v}}\varphi \geq 1$, where $\vec{v} \in T_{X,\tilde{x}}$ is the unique tangent direction at $\tilde{x}$ with $x \in U(\vec{v})$; then, after possibly replacing $\varphi$ with a smaller quasisubharmonic function, we may assume that $\varphi$ is linear of slope $-c$ on $[x,\tilde{x}]$, i.e.\ 
$$
\varphi(y) = \varphi(\tilde{x}) - c\left(\alpha(y) - \alpha(\tilde{x}) \right)
$$
for $y \in [x,\tilde{x}]$.
If $\gamma \coloneqq \lfloor c \rfloor$, and $\rho \coloneqq \rho_0 - \gamma\delta_{x_G}$, then the function $\phi(y) \coloneqq \varphi(y) + \gamma  \alpha(y \vee x)$ is $\rho$-subharmonic and $\phi(x_G) = 0$. 
Suppose there exists $\epsilon_1 > 0$ and a $(\phi,\epsilon_1)$-extension $\tilde{f} \in k[T]$ at $z$. 

Pick $u \in (x,\tilde{x})$ of type 2 such that $\alpha(x) - \alpha(u) \in (0,\eta)$, where $\eta$ is a positive number chosen to be less than $\frac{A(x)}{c }$ and, if $\phi(x) < 0$, then we also require that $\eta < \frac{-\epsilon_1 \phi(x)}{c }$. 
%
%

Pick $u' \in X^{\rig}$ such that $\mathfrak{r}_{\Gamma_{\varphi,n}}(u') = u$, and a polynomial generator $g$ of $\frakm_{u'}$ with $|g(x_G)| = 1$. 
Set $f \coloneqq \tilde{f} g^{\gamma}$. 
By construction, $z \vee u = z \vee x$, and hence we have
$$
|f(z)|e^{-\varphi(z)} = |g(z)|^{\gamma} |\tilde{f}(z)|e^{\gamma\alpha(z \vee u)} e^{-\phi(z)} = |\tilde{f}(z)|e^{-\phi(z)}.
$$

\emph{Case 1}. Suppose that $c$ is an integer and $\phi(x) = 0$. Then, $\phi |_{[x,x_G]} = 0$, and hence $\varphi(y) = -c\alpha(y \vee x)$ for $y \in [x,x_G]$.
Arguing as in Case 1 of~\cref{lemma:two_variant}, it suffices to find an $\epsilon_0> 0$ such that
\begin{equation}
\label{lemma 7 case 1}
|f(y)|e^{-(1+\epsilon_0)\varphi(y)-A(y)} \leq |\tilde{f}(z)|e^{-\phi(z)}
\end{equation}
for $y \in \Gamma_{\varphi,n}$. By assumption, $\Gamma_{\phi,n} \cap [x,x_G] \subseteq \{ x_G \}$, so the left-hand side of~\cref{lemma 7 case 1} is equal to $|\tilde{f}(y)|e^{-(1+\epsilon_0)\phi(y) - A(y)}$ for $y \in \Gamma_{\phi,n}$, in which case~\cref{lemma 7 case 1} holds. 
Thus, it suffices to verify~\cref{lemma 7 case 1} for $y \in [x,x_G]$. It holds when $y =x_G$ by assumption, and the function $y \mapsto |f(y)|e^{-(1+\epsilon_0)\varphi(y)-A(y)}$ is decreasing on $[u,x_G]$, but increasing on $[u,x]$. 
Therefore, it is enough to check~\cref{lemma 7 case 1} at $y=x$. 
Let $\beta \coloneqq \log |\tilde{f}(x)| - \log |\tilde{f}(x_G)| \leq 0$ and suppose there exists an $\epsilon_0 > 0$ such that
\begin{equation}
\label{lemma 7 case 1 2}
e^{\beta} |g(x)|^{c}e^{-(1+\epsilon_0)\varphi(x) - A(x)} \leq 1.
\end{equation}
Then, $|f(x)|e^{-(1+\epsilon_0)\varphi(x) - A(x)} \leq |\tilde{f}(x_G)|$, and $|\tilde{f}(x_G)|$ is less than or equal to $|\tilde{f}(z)|e^{-\phi(z)}$ by assumption. Thus, it suffices to show~\cref{lemma 7 case 1 2}. Observe that
\begin{align*}
e^{\beta} |g(x)|^{c} e^{-(1+\epsilon_0)\varphi(x) - A(x)} &= \textrm{exp}\left( \beta -c \alpha(u) - (1+\epsilon_0)\varphi(\tilde{x}) + (1+\epsilon_0)c \left( \alpha(x) - \alpha(\tilde{x}) \right) - A(x) \right)\\
&\leq \textrm{exp}\left( -c\left( \alpha(u) - \alpha(x) \right) - A(x) + \epsilon_0 c\alpha(x) \right),
\end{align*}
so~\cref{lemma 7 case 1 2} holds at $y=x$ if $\epsilon_0 \leq \frac{A(x) -c\eta}{c\alpha(x)}$.

\emph{Case 2}. Suppose that $c$ is an integer and $\phi(x) < 0$. As in Case 1 of~\cref{lemma:two_variant}, it suffices to find $\epsilon_0 > 0$ such that~\cref{lemma 7 case 1} holds
for all $y \in \Gamma_{\varphi,n}$. Set $\widetilde{\Gamma}_{\phi,n} = \Gamma_{\phi,n} \cup \{ x_G \}$. 
Again arguing as in Case 1 of~\cref{lemma:two_variant},~\cref{lemma 7 case 1} holds for $y \in \widetilde{\Gamma}_{\phi,n}$ provided $\epsilon_0 \leq \frac{\epsilon_1 n}{n+c(n+1)}$. Finally, as in Case 1, the function $y \mapsto |f(y)|e^{-(1+\epsilon_0)\varphi(y)-A(y)}$ on $\Gamma_{\varphi,n}\backslash \widetilde{\Gamma}_{\phi,n}$ is maximized at $y=x$, so it suffices to find an $\epsilon_0 > 0$ so that holds~\cref{lemma 7 case 1} there.
Observe that
$$
|g(u)|^c e^{(\epsilon_1 - \epsilon_0) \phi(x) + (1+\epsilon_0) c \alpha(x)} \leq 1
$$
provided $\epsilon_0 \leq \frac{-\epsilon_1 \phi(x) - c\eta}{-\phi(x) +c\alpha(x)}$, where $-\epsilon_1 \phi(x) - c\eta > 0$ by the choice of $u$. 
It follows that
$$
|f(x)|e^{-(1+\epsilon_0)\varphi(x) - A(x)} = |g(u)|^c e^{(\epsilon_1 - \epsilon_0)\phi(x)} e^{(1+\epsilon_0)c\alpha(x)} |\tilde{f}(x)| e^{-(1+\epsilon_1)\phi(x) - A(x)} \leq |\tilde{f}(z)|e^{-\phi(z)},
$$
which completes the proof in Case 2.

\emph{Case 3}. Suppose $c > \gamma$. 
Following the previous two cases, it suffices to find $\epsilon_0 > 0$ such that~\cref{lemma 7 case 1} holds for all $y \in \Gamma_{\varphi,n}$.
Set $\widetilde{\Gamma}_{\phi,n} = \Gamma_{\phi,n} \cup \{ x_G \}$.
If $y \in \widetilde{\Gamma}_{\phi,n}$, then $\phi(y) \leq -\frac{n}{n+1} \alpha(y)$, and hence
$$
\textrm{exp}\left( (\epsilon_1 - \epsilon_0)\phi(y) + \gamma \epsilon_0 \alpha(y \vee u') \right) \leq \left( e^{-\alpha(y)} \right)^{\frac{n}{n+1}(\epsilon_1-\epsilon_0) - \gamma  \epsilon_0}.
$$
This last quantity is less than or equal to $1$ provided $\epsilon_0 \leq \frac{\epsilon_1 n}{\gamma  (n+1) + n}$. In this case, we have
$$
|f(y)|e^{-(1+\epsilon_0)\varphi(y)-A(y)} \leq e^{\gamma  \epsilon_0 \alpha(y \vee u') + (\epsilon_1 - \epsilon_0)\phi(y)} |\tilde{f}(y)|e^{-(1+\epsilon_1 )\phi(y) - A(y)} \leq |\tilde{f}(z)| e^{-\phi(z)}.
$$
Let $x'' \in [u,x_G]$ be the minimal point of $\widetilde{\Gamma}_{\phi,n} \cap [u,x_G]$.
If $\vec{v}'' \in T_{X,x''}$ is the unique direction at $x''$ with $x \in U(\vec{v}'')$, then the function $y \mapsto |f(y)|e^{-(1+\epsilon_0)\varphi(y)-A(y)}$ is decreasing on $[u,x'']$ provided 
$$
d_{\vec{v}''}( \epsilon_0 \gamma  \alpha - (1+\epsilon_0)\phi - A ) \leq 0.
$$
Using the fact that $d_{\vec{v}''} \phi \geq -\frac{n}{n+1}$, this holds if $\epsilon_0 \leq \frac{1}{(n+1)\gamma +n}$.

Now, the function $y \mapsto |f(y)|e^{-(1+\epsilon_0)\varphi(y) - A(y)}$ on $[x,u)$ achieves its maximum at $y=x$, so it suffices to find an $\epsilon_0 > 0$ such that~\cref{lemma 7 case 1} holds at $y=x$. 
If $\phi(x) = 0$, one can argue as in Case 1. If $\phi(x) < 0$, then 
it follows that
$$
|g(u)|^{\gamma} e^{(\epsilon_1 - \epsilon_0)\phi(x) + (1+\epsilon_0)\gamma \alpha(x)} \leq 
\textrm{exp}\left( \gamma  \eta + \epsilon_1 \phi(x) -\epsilon_0 \varphi(x) \right),
$$
and this is bounded above by $1$ 
provided $\epsilon_0 \leq \frac{-\phi(x) \epsilon_1 - \gamma \eta}{-\varphi(x)}$; note that $-\phi(x)\epsilon_1 - \gamma  \eta > 0$ by the choice of $u$. 
Therefore, we have 
$$
|f(x)|e^{-(1+\epsilon_0)\varphi(x)-A(x)} \leq |g(u)|^{\gamma} e^{(\epsilon_1 - \epsilon_0)\phi(x) + \gamma(1+\epsilon_0)\alpha(x)} |\tilde{f}(x)|e^{-(1+\epsilon_1)\phi(x)-A(x)}  \leq |\tilde{f}(z)|e^{-\phi(z)},
$$
which completes the proof.
\end{proof}

If none of the hypotheses of the previous 4 lemmas hold, then $\Gamma_{\varphi,n}'$ must be the interval $[z,x_G]$. 
This special case is addressed below.
\begin{lemma}\label{lemma:five_variant}
Suppose $\mass(\rho_0) \geq 1$ and $\Gamma_{\varphi,n}' = [z,x_G]$, then there exists $\epsilon_0 > 0$ such that any nonzero constant function is a $(\varphi,\epsilon_0)$-extension at $z$. 
\end{lemma}

\begin{proof}
It suffices to find $\epsilon_0 > 0$ such that $-(1+\epsilon_0)\varphi(y) -A(y) \leq -\varphi(z)$ for any $y \in [z,x_G]$. 
For such $y$, $\varphi(z) - \varphi(y) \leq 0$ and hence it suffices to find $\epsilon_0 > 0$ such that $-\epsilon_0 \varphi(y) - A(y) \leq 0$ for all such $y$.
This holds for any $\epsilon_0 > 0$ at $y= x_G$.
Thus, if $\vec{v} \in T_{X,x_G}$ is the unique direction at $x_G$ with $z \in U(\vec{v})$, then it suffices to find $\epsilon_0 > 0$ such that $d_{\vec{v}}\left( - \epsilon_0 \varphi - A \right) \leq 0$. This holds if $\epsilon_0 \leq \frac{m(\vec{v})}{-d_{\vec{v}} \varphi}$. 
\end{proof}

Let us now summarize the proof of~\cref{thm:otvariant}: 
if $\mass(\rho_0) < 1$, then a $(\varphi,\epsilon_0)$-extension at $z$ exists by \cref{lemma:one_variant}. Suppose now that $\mass(\rho_0) \geq 1$; in particular, $\Gamma_{\varphi,n} \not= \emptyset$ and $\Ends(\Gamma_{\varphi,n}) \not= \emptyset$. As discussed at the start of the section, we assume that $\varphi$ is locally constant off of the convex hull of $\Gamma_{\varphi} \cup \{ z \}$. By repeatedly applying \cref{lemma:two_variant},  \cref{lemma:three_variant}, and \cref{lemma:four_variant}, we may assume that $\Ends(\Gamma_{\varphi,n})$ consists of a single type-2 or type-3 point onto which $z$ retracts.
This is precisely the setting of \cref{lemma:five_variant}, which then asserts that a $(\varphi,\epsilon_0)$-extension at $z$ exists. Now, if $z \in X^{\rig}$ and $k = k^a$, then $\H(z) = k$, so the second assertion is immediate from the first. This concludes the proof of~\cref{thm:otvariant}.

\section{Non-Archimedean Demailly approximation of quasisubharmonic Functions}\label{section:application}

Let $X$ be the Berkovich closed unit disc over $k$. Given a quasisubharmonic $\varphi$ on $X$, one may wish to approximate it by a sequence of quasisubharmonic functions whose singularities are controlled. One such result is already well known:~\cite[Theorem 2.10]{dynberko} shows that there is a decreasing sequence of bounded quasisubharmonic functions $(\varphi_n)_{n=1}^{\infty}$ on $X$ which decrease pointwise to $\varphi$. Let us briefly recall the construction: if $\varphi$ is $\rho_0$-subharmonic, for each $n \geq 1$, consider the finite subtree
$$
\Gamma_n \coloneqq \bigg\{ x \in X \colon (\rho_0 +\Delta \varphi) \{  y \geq x \} \geq 2^{-n} \textrm{ and } d(x_G,x) \leq 2^n \bigg\},
$$
where $d$ is the generalized metric on $X$ from~\S\ref{section:quasisubharmonic}. If $\mathfrak{r}_n \colon X \to \Gamma_n$ denotes the retraction map of $X$ onto $\Gamma_n$, then $\varphi_n$ is, up to translation by a constant, equal to $(\mathfrak{r}_n)^*\varphi$. A similar argument appears in~\cite[\S4.6]{favre06}.

Notice that the construction of the sequence $(\varphi_n)_{n=1}^{\infty}$ heavily uses the tree structure on $X$ (indeed, the same proof yields an analogous result for any metric tree) and moreover it depends on the choice of exhausting sequence of subtrees $\Gamma_n$. Without involving the analytic structure, it is unlikely that such a regularization result will generalize to higher-dimensional analytic spaces. The goal of this section is to construct a canonical regularization of a quasisubharmonic function on the Berkovich unit disc. The inspiration is the much-celebrated regularization theorem of Demailly~\cite[Proposition 3.1]{demailly92}.

Let us briefly recall the construction of the Demailly approximation of a plurisubharmonic function on the complex unit disc. Let $\mathbf{D}$ be the open unit disc in $\mathbf{C}$, and let $\varphi$ be a plurisubharmonic function on $\mathbf{D}$. Consider the Hilbert space $\mathscr{H}_{\varphi}$ of holomorphic functions $f$ on $\mathbf{D}$ satisfying the integrability condition
$$
\| f \|_{\varphi}^2 \coloneqq \int_{\mathbf{D}} |f|^2 e^{-2\varphi} d\lambda < +\infty,
$$
where $d\lambda$ is the Lebesgue measure. For each $m \geq 1$, let $(f_{m, n})_{n=1}^{\infty}$ be an orthonormal basis of $\mathscr{H}_{m\varphi}$ and define a plurisubharmonic function $\varphi_m$ on $\mathbf{D}$ by the formula 
$$
\varphi_m = \frac{1}{2m} \log\left( \sum_{n=1}^{\infty} |f_{m,n}|^2 \right).
$$
This function $\varphi_m$ is called the \emph{Demailly approximation} associated to $\mathscr{H}_{m\varphi}$; $\varphi_m$ has analytic singularities, 
in the sense that it can be locally written as $c \log \left( \sum_{i=1}^N |g_i |^2 \right) + \beta$ for some constant $\alpha > 0$,  local holomorphic functions $g_i$, and a locally bounded function $\beta$. The Demailly approximation may also be expressed as 
$$
\varphi_m = \sup_f \frac{1}{m} \log |f|,
$$
where the supremum ranges over all holomorphic functions $f$ in the unit ball of $\mathscr{H}_{m\varphi}$. In~\cite[Proposition 3.1]{demailly92}, it is shown that the sequence $(\varphi_m)_{m =1}^{\infty}$ converges pointwise (and in $L^1_{\textrm{loc}}$) to $\varphi$ and, after passing to a subsequence, is decreasing in $m$.

In~\S\ref{section:regularization}, given a quasisubharmonic function $\varphi$ on the Berkovich unit disc $X$, we construct an ideal $\H_{\varphi}$ of the Tate algebra and the non-Archimedean Demailly approximation $\varphi_m$ associated to $\H_{m\varphi}$. Using \cref{thm:otvariant}, we show that the sequence $(\varphi_m)_{m=1}^{\infty}$  converges to $\varphi$. In~\S\ref{section:multiplier}, we briefly discuss the connection between the ideals $\H_{\varphi}$ and a locally-defined non-Archimedean multiplier ideal associated to $\varphi$.

\subsection{Non-Archimedean Demailly approximation and a regularization theorem}\label{section:regularization}
Let $\varphi$ be a quasisubharmonic function on $X$ with $\varphi \leq 0$. The case when $\varphi$ may be positive will be addressed later. As in~\S\ref{section:ot}, consider the function $\| \cdot \|_{\varphi} \colon k\{ T \} \to [0,+\infty]$ defined by
$$
\| f \|_{\varphi} \coloneqq \sup_{x \in X \backslash Z(\varphi)} |f(x)|e^{-\varphi(x) - A(x)} \textrm{ , $f \in k\{ T \}$.}
$$
This gives a non-Archimedean norm on the subset $\large\{ \| \cdot \|_{\varphi} < +\infty \large\}$. The norms of this form are not submultiplicative in general, but they do satisfy a monotonicity property: if $\varphi,\phi$ are quasisubharmonic and such that $\varphi \leq \phi \leq 0$, then $\| f \|_{\phi} \leq \| f \|_{\varphi}$ for any $f \in k\{ T \}$. In particular, the limit 
$$
\| f \|^+_{\varphi} \coloneqq \lim_{\epsilon \to 0^+} \| f \|_{(1+\epsilon)\varphi}
$$
exists (though it may be infinite). Consider the subset $\H_{\varphi}$ of the Tate algebra consisting of those series $f \in k\{ T \}$ such that $\| f \|^+_{\varphi} < +\infty$. The function $\| \cdot \|^+_{\varphi}$ defines a non-Archimedean norm on $\H_{\varphi}$, which is not submultiplicative in general. 
With this norm, $\H_{\varphi}$ is a normed $k\{ T \}$-module. 

\begin{lemma}
The subset $\H_{\varphi}$ is a principal ideal of $k\{ T \}$, which is complete for the norm $\| \cdot \|^+_{\varphi}$. In particular, $\H_{\varphi}$ is a Banach $k\{ T \}$-module.
\end{lemma}

\begin{proof}
It is clear that $\H_{\varphi}$ is closed under addition. 
Given $g \in k\{ T \}$, the image of the function $x \mapsto |g(x)|$ is contained in the interval $[0,|g(x_G)|]$. 
For any $f \in \H_{\varphi}$ and for any $\epsilon > 0$, $\| fg \|_{(1+\epsilon)\varphi} \leq |g(x_G)| \cdot\| f \|_{(1+\epsilon)\varphi}$. In particular, $fg \in \H_{\varphi}$. Furthermore, the Tate algebra $k\{ T \}$ is a principal ideal domain (when $k$ is nontrivially-valued, see~\cite[Corollary 2.2.10]{bosch}; otherwise, $k\{ T \} = k[T]$) and hence $\H_{m\varphi}$ is principal. 

Let $(f_j)_{j=1}^{\infty} \subset \H_{\varphi}$ be a Cauchy sequence in the norm $\| \cdot \|^+_{\varphi}$. 
It follows that $(f_j)_{j=1}^{\infty}$ is also a Cauchy sequence for the Gauss norm, since $|\cdot |_{x_G} \leq \| \cdot \|_{\varphi}^+$. 
By the completeness of $k\{ T \}$, the sequence $(f_j)_{j=1}^{\infty}$ admits a limit $f \in k\{ T \}$. 
All ideals of $k\{ T \}$ are closed, so $f \in \H_{\varphi}$.
If $\H_{\varphi} = (h)$, then we can write $f_j = g_j h$ and $f = gh$ for some $g_j,g \in k\{ T \}$ such that $g_j \to g$ in $|\cdot |_{x_G}$. For any $\delta > 0$, take $\epsilon > 0$ sufficiently small so that $\| h \|_{(1+\epsilon)\varphi} < +\infty$ and take $j \gg 0$ so that $|(g_j - g)(x_G)| < \frac{\delta}{\| h \|^+_{\varphi}}$. It follows that
$$
\| f_j - f \|^+_{\varphi} \leq \| f_j - f \|_{(1+\epsilon)\varphi} \leq \left( \sup_{x \in X \backslash Z(\varphi)} |(g_j - g)(x)| \right) \| h \|_{(1+\epsilon)\varphi} = |(g_j - g)(x_G)| \cdot \| h \|_{(1+\epsilon)\varphi} < \frac{\delta}{\| h \|^+_{\varphi}}  \| h \|_{(1+\epsilon)\varphi} \leq \delta.
$$
Therefore, $f_j \to f$ in the norm $\| \cdot \|^+_{\varphi}$, and hence $\H_{\varphi}$ is complete with respect to $\| \cdot \|^+_{\varphi}$. 
\end{proof}

\begin{proposition}\label{prop:bounded}
Let $\varphi,\phi$ be two quasisubharmonic functions on $X$ with $\varphi ,\phi \leq 0$. 
If there exist $C_1,C_2 > 0$ such that $\phi - C_2 \leq \varphi \leq \phi + C_1$ on $X$, then
$$
e^{-C_1} \| \cdot \|_{\phi}^+ \leq \| \cdot \|_{\varphi}^+ \leq e^{C_2} \| \cdot \|_{\phi}^+
$$
as functions on $k\{ T \}$.
In particular, $\H_{\varphi}$ and $\H_{\phi}$ coincide as ideals of $k\{ T \}$ and the identity map between them is an isomorphism of Banach $k\{ T\}$-modules. 
Moreover, if $\varphi$ is bounded, then $\H_{\varphi} = k\{ T \}$.
\end{proposition}

The proof of~\cref{prop:bounded} is elementary. However,~\cref{prop:bounded} may be used to define $\H_{\varphi}$ for a quasisubharmonic function $\varphi$ such that $\sup_X \varphi > 0$, in which case it is not clear that the limit defining the norm $\| \cdot \|_{\varphi}^+$ exists.

The ideals $\H_{\varphi}$ satisfy the subadditivity-type property below.

\begin{proposition}\label{prop:subadditivity}
If $\varphi,\phi$ are quasisubharmonic functions on $X$, then $\H_{\varphi + \phi} \subseteq \H_{\varphi} \H_{\phi}$.
\end{proposition}

The proof of~\cref{prop:subadditivity} requires studying the extension of the ideal $\H_{\varphi}$ to each local ring $\O_{X,x}$ of $X$. This is discussed in~\S\ref{section:multiplier}. This proof is simpler than in the complex setting, where the proof of the subadditivity theorem for multiplier ideals relies crucially on the Ohsawa--Takegoshi theorem on the unit bidisc.

We will now define the non-Archimedean analogue of a plurisubharmonic function with analytic singularities. 

\begin{definition}
A quasisubharmonic function $\varphi$ on $X$ is said to have \emph{analytic singularities} if there exists a cover $\{ V_{i} \}_{i \in I}$ of $X$ by affinoid domains such that for each $i \in I$, there are analytic functions $f_{i,1},\ldots,f_{i,n_{i}} \in \O_X(V_{i})$, positive numbers $\alpha_{i,1},\ldots,\alpha_{i,n_i} > 0$, and a bounded function $\beta_i \colon V_i \to \R$ such that 
$$
\varphi|_{V_i} = \beta_i + \sum_{j=1}^{n_i} \alpha_{i,j} \log |f_{i,j}| \textrm{\hspace{12pt} on $V_i$.}
$$
\end{definition}

Quasisubharmonic functions with analytic singularities, like their complex counterparts, are quite well-behaved. One instance of this is illustrated below.

\begin{example}
Let $\varphi$ be a quasisubharmonic function on $X$ with analytic singularities and suppose it admits a decomposition $\varphi = \beta + \alpha \log |f|$, with $f \in k\{ T \}$ irreducible, $\alpha > 0$, and $\beta \colon X \to \R$ bounded. Then, $\H_{\varphi} = (f^{\lfloor \alpha \rfloor})$.
\end{example}

Given a quasisubharmonic function $\varphi$ on $X$ and a positive integer $m \geq 1$, the \emph{non-Archimedean Demailly approximation} associated to $\H_{m\varphi}$ is the function
$$
\varphi_m \coloneqq \frac{1}{m} \left( \sup_{f \in \H_{m\varphi} \backslash \{ 0 \}} \log \frac{|f|}{\| f \|^+_{m\varphi}} \right)^*,
$$
where $( - )^*$ denotes the upper-semicontinuous regularization.
When $k$ is nontrivially-valued, $\varphi_m$ may equivalently be defined in terms of a supremum over $f \in B(1)_{m\varphi} \coloneqq \{ f \in \H_{m\varphi} \colon \| f \|^+_{m\varphi} \leq 1 \}$, the unit ball in $\H_{m\varphi}$. This mirrors the definition of the Demailly approximation in the complex setting.

\begin{proposition}
Let $\varphi$ be a quasisubharmonic function on $X$ and let $m \geq 1$. The non-Archimedean Demailly approximation $\varphi_m$ is quasisubharmonic with analytic singularities.
\end{proposition}

\begin{proof}
For each $f \in \H_{m\varphi}$, $\frac{1}{m} \log \frac{|f|}{\| f \|^+_{m\varphi}} \leq \varphi(x_G)$ and hence by \cref{lemma:baker_rumely}, the function $\varphi_m$ is quasisubharmonic. Let $h \in \H_{m\varphi}$ be any generator. To show $\varphi_m$ has analytic singularities, it suffices to show that the function
$$
\beta(x) \coloneqq \frac{1}{m}\left(\sup_{f \in \H_{m\varphi} \backslash \{ 0 \}} \log \frac{|f(x)|}{\| f \|^+_{m\varphi}} \right)^*- \frac{1}{m} \log \frac{|h(x)|}{\| h \|^+_{m \varphi}} 
$$
is bounded. As $\| g h \|^+_{m\varphi} \geq |g(x_G)| |h(x_G)| e^{-m\varphi(x_G)} \not= 0$ for any nonzero $g \in k\{ T \}$, we have 
$$
\sup_{g \in k\{ T \}\backslash \{ 0 \}} \log \frac{|h(x)| |g(x)|}{\| g h \|^+_{m\varphi}} \leq \log \frac{|h(x)|}{|h(x_G)|} + m\varphi(x_G). 
$$
This upper bound is upper semicontinuous in $x$, so it follows that $\beta(x) \leq \frac{1}{m} \log \frac{\| h \|^+_{m\varphi}}{|h(x_G)|} + \varphi(x_G)$. 
For a lower bound, taking $g = 1$ gives $\beta \geq 0$.
Note that the decomposition $\varphi_m = \beta + \frac{1}{m}\log \frac{|h(x)|}{\| h \|^+_{m\varphi}}$ is not unique -- it depends on the choice of generator $h$ of the ideal $\H_{m\varphi}$.
\end{proof}

\begin{remark}
It is not true in general that, if $\varphi$ has analytic singularities, then $\varphi_m = \varphi$ for all $m \geq 1$. Rather, $\varphi_m$ is an ``algebraic approximation'' to $\varphi$, as the following example demonstrates. Let $\varphi = \alpha \cdot \log |f|$ for some $\alpha > 0$ and $f \in k\{ T \}$ irreducible. It is easy to verify that $\varphi_m = \frac{\lfloor m\alpha \rfloor}{m} \log |f|$. If $x \in X^{\rig}$ is the rigid point corresponding to the maximal ideal $(f)$ of $k\{ T\}$, then $\varphi_m$ has rational slope along the branch $[x,x_G]$ of $X$; however, this will not be the case for $\varphi$ if $\alpha \in \R \backslash \Q$.
\end{remark}

\begin{example}\label{example:not_decreasing}
The sequence $(\varphi_m)_{m = 1}^{\infty}$ does not, in general, decrease monotonically in $m$.
If $\varphi = \frac{3}{2} \log |T|$, then
$$
\varphi_m = \begin{cases}
\varphi, & m = 2n\\
\frac{3n+1}{2n+1} \log |T|, & m = 2n+1.
\end{cases}
$$
In particular, $\varphi_{2n} = \varphi$ for all $n \geq 0$, but the subsequence $(\varphi_{2n+1})_{n\geq 0}$ decreases monotonically to $\varphi$. 
It is not clear whether or not the sequence $\varphi_m$ admits a decreasing subsequence in general.
For related results in the complex case, see~\cite{dano-kim}.
\end{example}

The following 
is a non-Archimedean analogue of Demailly's regularization theorem~\cite[Proposition 3.1]{demailly92}.

\begin{theorem}\label{thm:convergence}
Assume $k$ is algebraically closed, trivially-valued, or is spherically complete of residue characteristic zero. 
Let $\varphi$ be a quasisubharmonic function on $X$ with $\varphi \leq 0$, and let $m \geq 1$. For any $x \in X$, 
$$
\varphi(x) \leq \varphi_m(x) \leq \varphi(x) + \frac{1}{m} A(x).
$$
In particular, $\varphi_m$ converges pointwise to $\varphi$ on $\{ A < +\infty \} \subseteq X$. 
\end{theorem}


The estimate in~\cref{thm:convergence} in fact yields a stronger assertion than just the pointwise convergence of the non-Archimedean Demailly approximation $\varphi_m$ to $\varphi$. Indeed, for any compact subset $K \subseteq X$ such that $\alpha |_K$ is bounded above, $\varphi_m$ converges uniformly to $\varphi$; this is reminiscent of the $L^1_{\textrm{loc}}$ convergence of the Demailly approximation in the complex case.
The estimate in~\cref{thm:convergence} can thus be thought of as asserting that $\varphi_m$ converges to $\varphi$ in a non-Archimedean version of the Hartog's sense, as introduced in~\cite[Definition 3.2.3]{dinh-sibony}.

The key ingredient in the proof of~\cref{thm:convergence} is \cref{thm:otvariant}. Moreover, the proof shows that the hypothesis that $\varphi \leq 0$ is only needed for the upper bound, whereas the lower bound always holds.

\begin{proof}
Consider first the upper bound on $\varphi_m$.
If $A(x) = +\infty$, then the upper bound is trivial.
Suppose that $A(x) < +\infty$, in which case $\alpha(x) < +\infty$ as well, so $\varphi_m(x)$ may be calculated without the upper-semicontinuous regularization by~\cref{lemma:baker_rumely}.
The upper bound then follows from the observation that
$$
m(\varphi_m - \varphi) - A = \sup_{f \in \H_{m\varphi}} \log \frac{|f|e^{-m\varphi-A}}{\| f \|^+_{m\varphi}} \leq 0,
$$
where we have used that $\varphi \leq 0$ to conclude that $|f|e^{-m\varphi - A} \leq \| f \|^+_{m\varphi}$. 

To get the lower bound, let $\phi \coloneqq \varphi - \varphi(x_G)$ and observe that $\H_{m\varphi} = \H_{m\phi}$ and $\| f \|^+_{m\phi} = e^{m\varphi(x_G)} \| f \|^+_{m\varphi}$ for any $f \in \H_{m\varphi}$. 
For any $x \in X$, \cref{thm:otvariant} asserts that there exists an $\epsilon_0 > 0$ and a $(m\phi,\epsilon_0)$-extension $f \in k[T]$ at $x$. In particular, $f \in \H_{m\phi}$ and 
$$
1 \leq \frac{|f(x)|}{\| f \|^+_{m\phi}} e^{-m\phi(x)} = \frac{|f(x)|}{\| f \|^+_{m\varphi}} e^{-m\varphi(x)},
$$
or equivalently $\varphi(x) \leq \frac{1}{m} \log \frac{|f(x)|}{\| f \|^+_{m\varphi}} \leq \varphi_m(x)$.
\end{proof}

\begin{remark}
A quasisubharmonic function cannot, in general, be bounded pointwise below by a quasisubharmonic function with analytic singularities. Indeed, a general quasisubharmonic function can have a dense set of poles in $X$: for example, if $\{ a_j \colon j \geq 0 \}$ forms a dense subset of $k^{\circ}$, set
$$\varphi \coloneqq \sum_{j=0}^{\infty} c_j \log |T-a_j|,$$
for some $c_j > 0$ satisfying $\sum_{j=0}^{\infty} c_j = 1$.\ However, a quasisubharmonic function with analytic singularities on $X$ must have a finite set of poles. This is analogous to the fact that a subharmonic function with analytic singularities on the complex unit disc has a discrete set of poles.
\end{remark}

\subsection{Non-Archimedean multiplier ideals}\label{section:multiplier}
Let $\mathcal{A} = k\{ T \}$ and $X=\M(\mathcal{A})$.\ The stalk $\O_{X,x}$ of the structure sheaf $\O_X$ at $x \in X$ is a noetherian local ring, which may be computed as the direct limit of the affinoid algebras $\mathcal{A}_V$ over all affinoid neighborhoods $V$ of $x$. The unique maximal ideal $\underline{\frakm}_x$ of $\O_{X,x}$ consists of those germs $f \in \O_{X,x}$ such that $|f(x)| = 0$. We say that a germ $f \in \O_{X,x}$ is \emph{defined on} an affinoid domain $V \subset X$ if it lies in the image of the natural map $\mathcal{A}_V \to \O_{X,x}$. 

For $x \in X^{\rig}$, let $\frakm_x$ denote the corresponding maximal ideal of $\mathcal{A}$. As in~\cite[7.3.2/1]{bgr}, $\underline{\frakm}_x$ is the extension of $\frakm_x$ along the natural map $\mathcal{A} \to \O_{X,x}$ (though this is false for arbitrary points of an affinoid space, as demonstrated by~\cite[Remark 2.2.9]{berkovich93}). Moreover, for $x \in X^{\rig}$, the natural map $\mathcal{A} \to \O_{X,x}$ factors through the localization $\mathcal{A}_{\frakm_x}$, which induces an isomorphism $\widehat{\mathcal{A}}_{\frakm_x} \stackrel{\sim}{\to} \widehat{\O}_{X,x}$ on completions.

In complex geometry, we associate to each plurisubharmonic function (or more generally, to a semipositive metric on a line bundle) a \emph{multiplier ideal sheaf}, which measures the singularities of the plurisubharmonic function. 
We propose a non-Archimedean analogue of this notion.

\begin{definition}
Let $\varphi$ be a quasisubharmonic function on $X$ with $\varphi \leq 0$. 
For any $x \in X$, the \emph{local multiplier ideal} of $\varphi$ at $x$, denoted $\J(\varphi)_x$, consists of those germs $f \in \O_{X,x}$ such that for all $\epsilon > 0$, there is an affinoid neighborhood $V_{\epsilon}$ of $x$ on which $f$ is defined such that 
$$
\sup_{V_{\epsilon} \backslash Z(\varphi)} |f|e^{-(1+\epsilon)\varphi - A} < +\infty.
$$
\end{definition}

It is easy to check that the local multiplier ideal $\J(\varphi)_x$ is indeed an ideal of the local ring $\O_{X,x}$. Furthermore, it is immediate from the definition that the local multiplier ideals satisfy $\J(\varphi + \beta)_x = \J(\varphi)_x$ where $\beta$ is a bounded usc function on $X$. In particular, if $\sup_X \varphi > 0$, we may define the local multiplier ideal of $\varphi$ at $x$ to be $\J(\varphi - \varphi(x_G))_x$. For simplicity, we assume from now on that all quasisubharmonic functions are nonpositive.

The name of a multiplier ideal is justified by examples of the following form: if $f \in \mathcal{A}$ is irreducible, $c > 0$, and $\varphi \coloneqq c \cdot \log |f|$, then $\J(\varphi)_x = \underline{\frakm}_x^{\lfloor c \rfloor}$, where $x \in X^{\rig}$ is the rigid point corresponding to the maximal ideal $(f)$ of $\mathcal{A}$.  More generally, the local multiplier ideals admit a similar description for a general quasisubharmonic function, as is made precise in the following lemma.

\begin{lemma}\label{lemma:local}
For $x \in X^{\rig}$, $\J(\varphi)_x = \underline{\frakm}_x^{\lfloor c_x \rfloor}$, where $c_x \coloneqq  \frac{\Delta \varphi \{x \}}{m(x)}$. For $x \in X \backslash X^{\rig}$, $\J(\varphi)_x = \O_{X,x}$. 
\end{lemma}

The quantity $c_x$
 may be thought of as the ``non-Archimedean Lelong number'' of $\varphi$ at the point $x$.

\begin{proof}
Fix $x \in X^{\rig}$. For a germ $f \in \O_{X,x}$, let $\ord_x(f)$ denote the maximal power of $\underline{\frakm}_x$ to which $f$ belongs. To say $f \in \J(\varphi)_x$ is equivalent the existence of an $\epsilon > 0$ such that $\ord_x(f) - (1+\epsilon)c_x + 1 \geq 0$, which occurs if and only if $\ord_x(f) > c_x - 1$, i.e. $\ord_x(f) \geq \lfloor c_x \rfloor$. Thus, $\J(\varphi)_x = \underline{\frakm}_x^{\lfloor c_x \rfloor}$.

Fix $x \in X \backslash X^{\rig}$ and any $\epsilon_0 > 0$. There are only finitely-many rigid points $z_1,\ldots,z_n \in X^{\rig}$ satisfying $(1+\epsilon_0) c_{z_i} \geq 1$, so we may find an affinoid neighborhood $V$ of $x$ which avoids $z_1,\ldots,z_n$ and hence $\sup_{V\backslash Z(\varphi)} |1|e^{-(1 +\epsilon)\varphi - A} < +\infty$ for all $\epsilon \in [0,\epsilon_0]$. Thus, $1 \in \J(\varphi)_x$. 
\end{proof}

One can show that the local multiplier ideals $\J(\varphi)_x$ arise as the stalks of a coherent sheaf of ideals on $X$. More precisely, the local multiplier ideals satisfy a ``coherence'' property similar to~\cite[Proposition 5.7]{demailly12}, originally due to Nadel~\cite{nadel}.

\begin{lemma}\label{lemma:coherence}
For $x \in X^{\rig}$, $\J(\varphi)_x = \H_{\varphi} \cdot \O_{X,x}$. 
\end{lemma}

\begin{proof}
The inclusion $\H_{\varphi} \cdot \O_{X,x} \subseteq \J(\varphi)_x$ is clear. 
The local ring $\O_{X,x}$ is a dvr, so the ideal $\H_{\varphi} \cdot \O_{X,x}$ can be written as $\underline{\frakm}_x^d$ for some $d \in \Z_{\geq 0}$; the power $d$ is the largest integer such that $\H_{\varphi} \subseteq \frakm_x^d$. 
We first show that $d \geq \lfloor c_x \rfloor$:
pick a net $(x_{j})_{j \in J}$ in $\{ \alpha < +\infty \}$ that converges to $x$ such that $x_j \geq x$ for all $j \in J$. 
After shifting $\varphi$ by a constant, assume that $\varphi(x_G) = 0$, and arguing as in~\cite[Lemma 2.9]{dynberko} we have
\begin{equation}\label{eqn:coherence}
\varphi(x_j) = -\int_{x_G}^{x_j} (\Delta \varphi)\{ y \in X \colon y \leq x_j \} d\alpha(y) \leq -\int_{x_G}^{x_j} (\Delta \varphi)\{ x \} d\alpha(y) = -\Delta \varphi\{ x \} \alpha(x_j).
\end{equation}
Fix a generator $g_x$ of $\frakm_x$. 
For any $f \in \H_{\varphi}$, there is a unique factorization $f = g_x^a h$ (up to units) for some $a \in \Z_{\geq 0}$ and $h \in \mathcal{A} \backslash \frakm_x$. 
For sufficiently small $\epsilon > 0$,~\cref{eqn:coherence} implies that
\begin{align*}
+\infty > \sup_{X \backslash Z(\varphi)} |f|e^{-(1+\epsilon)\varphi-A} 
&\geq \sup_{j \in J} |f(x_j)| e^{-(1+\epsilon)\varphi(x_j) - A(x_j)} \\
&\geq \sup_{j \in J} |h(x_j)| e^{-am(x_j)\alpha(x_j) + (1+\epsilon) (\Delta \varphi \{ x \})\alpha(x_j) - m(x_j)\alpha(x_j)}\\
&\geq \sup_{j \in J} |h(x_j)| e^{\alpha(x_j)\left(-am(x) + (1+\epsilon) (\Delta \varphi \{ x \})- m(x) \right)}
\end{align*}
where we have used that $m(x) \geq m(x_j)$ and $m(x_j)\alpha(x_j) \geq A(x_j)$, which follows from~\cref{definition of alpha}.
This last supremum is finite only if $-am(x) + (1+\epsilon) (\Delta \varphi \{ x \})- m(x) \leq 0$, and this must hold for all $\epsilon > 0$ sufficiently small; said differently, $a \geq \lfloor c_x \rfloor$. 
Thus, $d \geq \lfloor c_x \rfloor$. 

Now, any generator $f$ of $\H_{\varphi}$, there is a unique way to write $f = g_x^d h$ (up to units) for some $h \in \mathcal{A} \backslash \frakm_x$. If $d > \lfloor c_x \rfloor$, then it is easy to check that $g_x^{\lfloor c_x \rfloor} h$ lies in $\H_{\varphi}$ (it is the same calculation as above near $x$, and hence $g_x^{\lfloor c_x \rfloor} h$ must be a generator of $\H_{\varphi}$; this is a contradiction, so $d = \lfloor c_x \rfloor$. 
Thus, $\H_{\varphi} \cdot \O_{X,x} = \underline{\frakm}_x^{\lfloor c_x \rfloor}$ and we apply~\cref{lemma:local} to conclude.
\end{proof}

It is not hard to check that the local multiplier ideals satisfy many of the usual properties of multiplier ideals on complex algebraic varieties, e.g.\ invariance under small perturbations, the behavior under adding integral divisors, and the subadditivity property (in the subsequent lemma). This further justifies the terminology.

\begin{lemma}\label{lemma:local_subadditivity}
If $\varphi,\phi$ are quasisubharmonic functions on $X$ and $x \in X$, then $\J(\varphi + \phi)_x \subseteq \J(\varphi)_x \J(\phi)_x$.
\end{lemma}

\begin{proof}
The assertion follows from~\cref{lemma:local} and the observation that $\lfloor \Delta (\varphi + \phi)\{ x \} \rfloor \geq \lfloor \Delta \varphi \{ x \} \rfloor + \lfloor \Delta \phi \{ x \} \rfloor$.
\end{proof}

\begin{proof}[Proof of~\cref{prop:subadditivity}]
After translating $\varphi$ and $\phi$ by constants, we may assume that they are nonpositive. The inclusion $\H_{\varphi + \phi} \subseteq \H_{\varphi} \H_{\phi}$ holds if and only if there is an inclusion $\H_{\varphi + \phi} \widehat{\mathcal{A}}_{\frakm_x} \subseteq \left( \H_{\varphi} \H_{\phi} \right) \widehat{\mathcal{A}}_{\frakm_x}$ of the extensions along $\mathcal{A} \to \widehat{\mathcal{A}}_{\frakm_x}$, for every maximal ideal $\frakm_x$ of $\mathcal{A}$. The claim then follows from~\cref{lemma:coherence} and~\cref{lemma:local_subadditivity}.
\end{proof}

As demonstrated by~\cref{lemma:local}, the local multiplier ideals on the Berkovich unit disc are somewhat degenerate, principally because the types of singularities of quasisubharmonic functions that can occur in one dimension are limited. In this paper, they serve only to prove~\cref{prop:subadditivity}. However, the definition naturally extends to higher dimensions and, in that more complicated setting, could prove quite useful.

\begin{remark}
Given a quasisubharmonic function $\varphi$ on $X$ and $x \in X$, it is easy to see that ideals $\H_{\varphi}$ and $\J(\varphi)_x$ satisfy an analogue of the openness conjecture. That is, $\H_{\varphi} = \bigcup_{\epsilon > 0} \H_{(1+\epsilon)\varphi}$ and similarly for $\J(\varphi)_x$. The openness conjecture, for a plurisubharmonic function defined on a neighborhood of the origin in $\C^n$, has been solved in~\cite{bo2013}.
\end{remark}


\begin{thebibliography}{BRG84}
 
 \bibitem[Ax70]{ax-zeros}
James Ax.
\newblock Zeros of polynomials over local fields---{T}he {G}alois action.
\newblock {\em J. Algebra}, 15:417--428, 1970.

\bibitem[Ber90]{berkovich}
Vladimir~G. Berkovich.
\newblock {\em Spectral theory and analytic geometry over non-{A}rchimedean
  fields}, volume~33 of {\em Mathematical Surveys and Monographs}.
\newblock American Mathematical Society, Providence, RI, 1990.

\bibitem[Ber93]{berkovich93}
Vladimir~G. Berkovich.
\newblock \'{E}tale cohomology for non-{A}rchimedean analytic spaces.
\newblock {\em Inst. Hautes \'Etudes Sci. Publ. Math.}, (78):5--161 (1994),
  1993.

\bibitem[Ber96]{berndtsson1996extension}
Bo~Berndtsson.
\newblock The extension theorem of {O}hsawa-{T}akegoshi and the theorem of
  {D}onnelly-{F}efferman.
\newblock {\em Ann. Inst. Fourier (Grenoble)}, 46(4):1083--1094, 1996.

\bibitem[Ber13]{bo2013}
Bo~Berndtsson.
\newblock The openness conjecture for plurisubharmonic functions.
\newblock {\em arXiv:1305.5781}.

\bibitem[BFJ08]{bfj08}
S{\'e}bastien Boucksom, Charles Favre, and Mattias Jonsson.
\newblock Valuations and plurisubharmonic singularities.
\newblock {\em Publications of the Research Institute for Mathematical
  Sciences}, 44(2):449--494, 2008.

\bibitem[BFJ15]{bfj15}
S{\'e}bastien Boucksom, Charles Favre, and Mattias Jonsson.
\newblock Solution to a non-{A}rchimedean {M}onge-{A}mp\`ere equation.
\newblock {\em J. Amer. Math. Soc.}, 28(3):617--667, 2015.

\bibitem[BFJ16]{bfj16}
S{\'e}bastien Boucksom, Charles Favre, and Mattias Jonsson.
\newblock Singular semipositive metrics in non-{A}rchimedean geometry.
\newblock {\em J. Algebraic Geom.}, 25(1):77--139, 2016.

\bibitem[BGR84]{bgr}
S.~Bosch, U.~G{\"u}ntzer, and R.~Remmert.
\newblock {\em Non-{A}rchimedean analysis}, volume 261 of {\em Grundlehren der
  Mathematischen Wissenschaften [Fundamental Principles of Mathematical
  Sciences]}.
\newblock Springer-Verlag, Berlin, 1984.
\newblock A systematic approach to rigid analytic geometry.

\bibitem[BJ16]{boucksom-jonsson}
S{\'e}bastien Boucksom and Mattias Jonsson.
\newblock Tropical and non-{A}rchimedean limits of degenerating families of
  volume forms.
\newblock {\em arXiv:1605.05277}. To appear in J.\ \'{E}c.\ polytech. Math.


\bibitem[B{\l}o13]{blocki}
Zbigniew B{\l}ocki.
\newblock Suita conjecture and the {O}hsawa-{T}akegoshi extension theorem.
\newblock {\em Invent. Math.}, 193(1):149--158, 2013.

\bibitem[Bos14]{bosch}
Siegfried Bosch.
\newblock {\em Lectures on formal and rigid geometry}, volume 2105 of {\em
  Lecture Notes in Mathematics}.
\newblock Springer, Cham, 2014.

\bibitem[BR10]{baker-rumely}
Matthew Baker and Robert Rumely.
\newblock {\em Potential theory and dynamics on the {B}erkovich projective
  line}, volume 159 of {\em Mathematical Surveys and Monographs}.
\newblock American Mathematical Society, Providence, RI, 2010.

\bibitem[CLD12]{clducros}
Antoine Chambert-Loir and Antoine Ducros.
\newblock Formes diff{\'e}rentielles r{\'e}elles et courants sur les espaces de
  {B}erkovich.
\newblock {\em arXiv:1204.6277}.

\bibitem[Dem92]{demailly92}
Jean-Pierre Demailly.
\newblock Regularization of closed positive currents and intersection theory.
\newblock {\em J. Algebraic Geom.}, 1(3):361--409, 1992.

\bibitem[Dem00]{demailly2000ohsawa}
Jean-Pierre Demailly.
\newblock On the {O}hsawa-{T}akegoshi-{M}anivel {$L^2$} extension theorem.
\newblock In {\em Complex analysis and geometry ({P}aris, 1997)}, volume 188 of
  {\em Progr. Math.}, pages 47--82. Birkh\"auser, Basel, 2000.

\bibitem[Dem12]{demailly12}
Jean-Pierre Demailly.
\newblock {\em Analytic methods in algebraic geometry}, volume~1 of {\em
  Surveys of Modern Mathematics}.
\newblock International Press, Somerville, MA; Higher Education Press, Beijing,
  2012.


\bibitem[DS09]{dinh-sibony}
Tien-Cuong Dinh and Nessim Sibony.
\newblock Super-potentials of positive closed currents, intersection theory and
  dynamics.
\newblock {\em Acta Math.}, 203(1):1--82, 2009.

\bibitem[FJ05]{favre-jonsson}
Charles Favre and Mattias Jonsson.
\newblock Valuative analysis of planar plurisubharmonic functions.
\newblock {\em Invent. Math.}, 162(2):271--311, 2005.

\bibitem[FRL06]{favre06}
Charles Favre and Juan Rivera-Letelier.
\newblock \'{E}quidistribution quantitative des points de petite hauteur sur la
  droite projective.
\newblock {\em Math. Ann.}, 335(2):311--361, 2006.

\bibitem[Gub98]{gubler1998local}
Walter Gubler.
\newblock Local heights of subvarieties over non-{A}rchimedean fields.
\newblock {\em J. Reine Angew. Math.}, 498:61--113, 1998.

\bibitem[Jon15]{dynberko}
Mattias Jonsson.
\newblock Dynamics of {B}erkovich spaces in low dimensions.
\newblock In {\em Berkovich spaces and applications}, volume 2119 of {\em
  Lecture Notes in Math.}, pages 205--366. Springer, Cham, 2015.

\bibitem[Ked11]{kedlaya}
Kiran~S. Kedlaya.
\newblock Semistable reduction for overconvergent {$F$}-isocrystals, {IV}:
  local semistable reduction at nonmonomial valuations.
\newblock {\em Compos. Math.}, 147(2):467--523, 2011.

\bibitem[Kim14]{dano-kim}
Dano Kim.
\newblock A remark on the approximation of plurisubharmonic functions.
\newblock {\em C. R. Math. Acad. Sci. Paris}, 352(5):387--389, 2014.

\bibitem[Man93]{manivel1993theoreme}
Laurent Manivel.
\newblock Un th\'eor\`eme de prolongement {$L^2$} de sections holomorphes d'un
  fibr\'e hermitien.
\newblock {\em Math. Z.}, 212(1):107--122, 1993.


\bibitem[MN15]{mustata-nicaise}
Mircea Musta{\cb{t}}{\u{a}} and Johannes Nicaise.
\newblock Weight functions on non-{A}rchimedean analytic spaces and the
  {K}ontsevich-{S}oibelman skeleton.
\newblock {\em Algebr. Geom.}, 2(3):365--404, 2015.

\bibitem[MV07]{mcneal-varolin}
Jeffery~D. McNeal and Dror Varolin.
\newblock Analytic inversion of adjunction: {$L^2$} extension theorems with
  gain.
\newblock {\em Ann. Inst. Fourier (Grenoble)}, 57(3):703--718, 2007.

\bibitem[Nad90]{nadel}
Alan~Michael Nadel.
\newblock Multiplier ideal sheaves and {K}\"ahler-{E}instein metrics of
  positive scalar curvature.
\newblock {\em Ann. of Math. (2)}, 132(3):549--596, 1990.

\bibitem[OT87]{ohsawa-takegoshi}
Takeo Ohsawa and Kensh{\=o} Takegoshi.
\newblock On the extension of {$L^2$} holomorphic functions.
\newblock {\em Math. Z.}, 195(2):197--204, 1987.

\bibitem[Sch15]{schmidt}
Tobias Schmidt.
\newblock Forms of an affinoid disc and ramification.
\newblock {\em Ann. Inst. Fourier (Grenoble)}, 65(3):1301--1347, 2015.

\bibitem[Siu96]{siu1996fujita}
Yum-Tong Siu.
\newblock The {F}ujita conjecture and the extension theorem of
  {O}hsawa-{T}akegoshi.
\newblock In {\em Geometric complex analysis ({H}ayama, 1995)}, pages 577--592.
  World Sci. Publ., River Edge, NJ, 1996.

\bibitem[Siu98]{siu1998invariance}
Yum-Tong Siu.
\newblock Invariance of plurigenera.
\newblock {\em Invent. Math.}, 134(3):661--673, 1998.

\bibitem[Tem16]{temkin}

Michael Temkin. 
\newblock {\em Metrization of Differential Pluriforms on Berkovich Analytic Spaces}. 
\newblock {In: Baker M., Payne S. (eds) \em{Nonarchimedean and Tropical Geometry}. Simons Symposia. Springer, Cham, 2016.}

\bibitem[Thu05]{thuillier2005theorie}
Amaury Thuillier.
\newblock {\em Th{\'e}orie du potentiel sur les courbes en g{\'e}om{\'e}trie
  analytique non archim{\'e}dienne. {A}pplications {\`a} la th{\'e}orie
  d'{A}rakelov}.
\newblock PhD thesis, Universit{\'e} Rennes 1, 2005.

\bibitem[Zha95]{zhang1995small}
Shouwu Zhang.
\newblock Small points and adelic metrics.
\newblock {\em J. Algebraic Geom.}, 4(2):281--300, 1995.

\end{thebibliography}
\end{document}